\pdfoutput=1        % see https://arxiv.org/help/submit_tex

\documentclass[final]{siamltex}         % siam format but wider page
\usepackage[top=26mm, bottom=26mm, left=28mm, right=28mm]{geometry}
\usepackage{graphicx,bm,amssymb,amsmath,color}
\usepackage{afterpage}

\newcommand{\bi}{\begin{itemize}}
\newcommand{\ei}{\end{itemize}}
\newcommand{\ben}{\begin{enumerate}}
\newcommand{\een}{\end{enumerate}}
\newcommand{\be}{\begin{equation}}
\newcommand{\ee}{\end{equation}}
\newcommand{\bea}{\begin{eqnarray}} 
\newcommand{\eea}{\end{eqnarray}}
\newcommand{\ba}{\begin{align}}
\newcommand{\ea}{\end{align}}
\newcommand{\bse}{\begin{subequations}} 
\newcommand{\ese}{\end{subequations}}
\newcommand{\bc}{\begin{center}}
\newcommand{\ec}{\end{center}}
\newcommand{\bfi}{\begin{figure}}
\newcommand{\efi}{\end{figure}}
\newcommand{\ca}[2]{\caption{#1 \label{#2}}}

\newcommand{\ig}[2]{\includegraphics[#1]{#2}}
\newcommand{\tbox}[1]{{\mbox{\tiny \rm #1}}}
\newcommand{\mbf}[1]{{\mathbf #1}}
\newcommand{\pO}{{\partial\Omega}}
\newcommand{\half}{\mbox{\small $\frac{1}{2}$}}
\newcommand{\uu}{\mbf{u}}                       % velocity
\newcommand{\xx}{\mbf{x}}                       % x vec
\newcommand{\yy}{\mbf{y}}                       % y vec
\newcommand{\rr}{\mbf{r}}                       % r vec
\newcommand{\ww}{\mbf{w}}           % left null of Q.
\newcommand{\ff}{\mbf{f}}           % f vec
\newcommand{\vv}{\mbf{v}}           % v vec
\newcommand{\nn}{n}  % {\mbf{n}}                       % normal
\newcommand{\Tv}{\mbf{T}}                      % traction vec
\newcommand{\nx}{\nn^\xx}                       % n_x   (Stokes case)
\newcommand{\ny}{\nn^\yy}                       % n_y
\newcommand{\dx}{d_\xx}                          % target dipole func
\newcommand{\dy}{d_\yy}                          % src dipole func
\newcommand{\vt}[2]{\left[\begin{array}{l}#1\\#2\end{array}\right]} % 2-col-vec
\newcommand{\mt}[4]{\left[\begin{array}{ll}#1&#2\\#3&#4\end{array}\right]} % 2x2

\newcommand{\RR}{\mathbb{R}}
\newcommand{\ex}{\mbf{e}_1}                      % lattice vectors
\newcommand{\ey}{\mbf{e}_2}
\newcommand{\cb}{{\cal B}}                     % unit cell
\newcommand{\pcb}{{\partial \cal B}}                     % unit cell bdry
\newcommand{\cbo}{{\cb\backslash \overline{\Omega}}}  % punctured unit cell
\newcommand{\emach}{{\varepsilon_\tbox{mach}}}
\newcommand{\bigO}{{\cal O}}
\newcommand{\btau}{\bm{\tau}}
\newcommand{\ttau}{\tilde{\bm{\tau}}}
\newcommand{\rp}{{R_p}}               % proxy radius
\newcommand{\tr}{^\tbox{T}}                % transp (don't confuse w/ traction)
\newcommand{\Gp}{G^\tbox{p}}              % pressure stokeslet kernel
\newcommand{\Gt}{G^\tbox{t}}              % traction stokeslet kernel
\newcommand{\Dp}{D^\tbox{p}}              % pressure DLP kernel
\newcommand{\Dt}{D^\tbox{t}}              % traction DLP kernel
\newcommand{\hhh}[1]{{\mbf{h}^{(#1)}}}         % the 3 density vectors
\newcommand{\rrr}[1]{{\mbf{r}^{(#1)}}}         % the 3 proxy vectors
\newcommand{\www}[1]{{\mbf{w}^{(#1)}}}         % the 3 consistency cond vectors
\DeclareMathOperator{\Nul}{Nul}
\DeclareMathOperator{\Ran}{Ran}

\newtheorem{thm}{Theorem}[section]  % fix me for SIREV ***

\newtheorem{lem}[thm]{Lemma}

\newtheorem{pro}[thm]{Proposition}
\newtheorem{rmk}[thm]{Remark}

\begin{document}
\title{A unified integral equation scheme for doubly-periodic Laplace and Stokes boundary value problems in two dimensions}

\author{Alex Barnett\thanks{Department of Mathematics, Dartmouth College,
and Center for Computational Biology, Flatiron Institute}
\and
Gary Marple\thanks{Department of Mathematics, University of Michigan}
\and
Shravan Veerapaneni\footnotemark[2]   % repeats the uMich affil
\and
Lin Zhao\thanks{INTECH, Princeton, NJ}
}

\date{\today}
\maketitle
\begin{abstract}
We present a spectrally-accurate scheme to
turn a boundary integral formulation for an elliptic PDE
on a single unit cell geometry
into one for the fully periodic problem.
Applications include computing the effective permeability %tensor
of composite media (homogenization), and microfluidic chip design.
%periodize integral equation solvers
%so that the number of dimensions of periodicity matches the space dimension,
%such as occurs in modeling the effective flow response
%(homogenization)
%of composite media.
%
Our basic idea is to exploit a small least squares solve to
apply periodicity without ever handling 
periodic Green's functions.
%without ever having to compute it explicitly.  % analytically ?
%
We exhibit fast solvers for the two-dimensional (2D)
doubly-periodic Neumann Laplace problem (flow around insulators),
and Stokes non-slip fluid flow problem, that for inclusions with smooth boundaries achieve
12-digit accuracy, and can handle thousands of inclusions per unit cell.
%(heat flow in a periodic array of insulators)
%
We split the infinite sum over the lattice of images
into a directly-summed ``near'' part plus a small number of auxiliary sources
which represent the (smooth) remaining ``far'' contribution.
Applying physical boundary conditions on the unit cell walls
gives an expanded linear system,
which, after a rank-1 or rank-3 correction and a Schur complement,
leaves a {\em well-conditioned}
square system which can be solved
iteratively %via GMRES,
using fast multipole acceleration plus a low-rank term.
We are rather explicit about the consistency and nullspaces
of both the continuous and discretized problems.
% and consistency conditions ?
%
The scheme is simple
%(no analytic formulae,
(no lattice sums,
Ewald methods, nor particle meshes are required),
allows adaptivity,
and is essentially dimension- and PDE-independent,
so would generalize without fuss to 3D and to other non-oscillatory elliptic
problems such as elastostatics.
We incorporate %and test
recently developed spectral %near-singular
quadratures that accurately handle close-to-touching %inclusion
geometries.
We include many numerical examples, and provide a software implementation.
% **** Some stuff about building on FMM allows non-uniform particle
%  distributions? high-aspect unit cells?
\end{abstract}

% IIIIIIIIIIIIIIIIIIIIIIIIIIIIIIIIIIIIIII
\section{Introduction}
\label{s:intro}

\bfi[t] % fffffffffffffffffffffffffffffffffffffffffffffffffffffffffffffffffff
\bc\ig{width=6.3in}{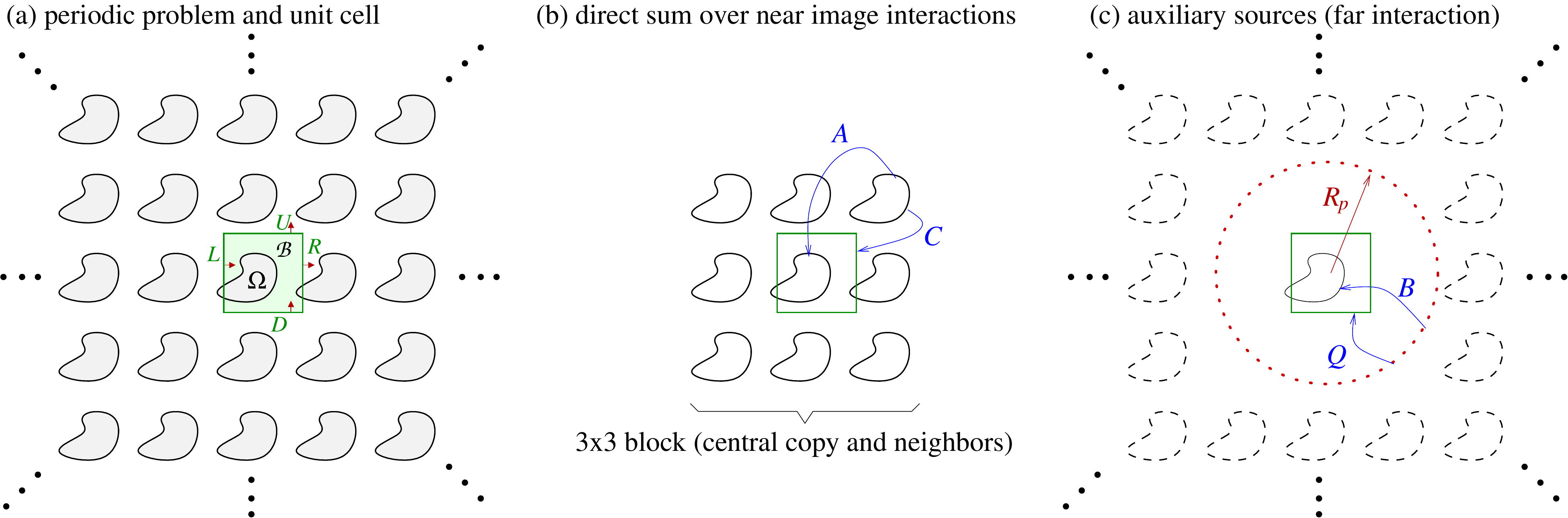}\ec
\ca{2D periodic problem in the case of a single inclusion
$\Omega$ with boundary $\pO$. (a) Periodic BVP in $\RR^2$, showing a possible unit cell ``box'' $\cb$ and its four walls $L$, $R$, $D$, $U$, and senses of wall normals.
(b) Directly-summed ``near'' copies of $\pO$.
(c) Circle of auxiliary sources (red dots)
which represent the field in $\cal B$ due to the
infinite punctured lattice of ``far'' copies (dashed curves).
In (b)--(c) blue arrows indicate the action (source to target) of
the four matrix blocks $A$, $B$, $C$, and $Q$.
}{f:geom}
\efi

Periodic boundary value problems (BVPs) arise frequently in
engineering and the sciences,
either when modeling the behavior of solid or fluid media
with true periodic geometry,
or when applying artificial periodic boundary conditions
to simulate a representative domain of a random medium or particulate flow
(often called a super-cell or representative volume element simulation).
The macroscopic response of a given microscopic periodic composite medium
can often be summarized by an effective material property
(e.g.\ a conductivity or permeability tensor),
a fact placed on a rigorous footing by the 
field of {\em homogenization} (for a review see \cite{cioranescu}).
However, with the exception of layered media that vary in only one dimension,
finding this %conductivity or permeability
tensor requires
the numerical solution of ``cell problems'' \cite{miltonbook,pavliotis},
namely BVPs in which the solution is
periodic up to some additive constant expressing
the macroscopic driving.
%The task of computing its effective macroscopic medium properties
%is often called {\em numerical homogenization} \cite{miltonbook,guedes};
%there is also a large literature on its mathematical analysis \cite{cioranescu}. 
Application areas span all of the major elliptic PDEs, including
the Laplace equation (thermal/electrical conductivity, electrostatics and magnetostatics of composites \cite{helsingdiscs,helsingzeta2,Moura94,cazeaux});
the Stokes equations (porous flow in periodic solids
\cite{drummond,larsonhigdon2,wangfibers,Krop04},
sedimentation \cite{klint}, mobility \cite{mobility}, transport by
cilia carpets \cite{carpets},
vesicle dynamics in microfluidic flows \cite{periodicp});
elastostatics (microstructured periodic or
random composites \cite{guedes,helsingelasto,elastohomog,Otani06});
and the Helmholtz and Maxwell equations (phononic and
photonic crystals, bandgap materials \cite{jobook,otani08}).
In this work we focus on the first two (non-oscillatory) PDEs above,
noting that the methods that we present also apply with minor changes
to the oscillatory Helmholtz and Maxwell cases,
at least up to moderate frequencies \cite{qplp,mlqp,acper}.
%which are {\em non-oscillatory}, noting that
%the oscillatory Helmholtz and Maxwell cases have
%added complications such as Wood anomalies \cite{shipmanreview}.

% REPEATS THE ABSTRACT:
%In this paper we present a unified scheme for converting an
%existing integral equation solver involving free-space
%(non-periodic) kernels, to one which is periodic in all
%space dimensions, without having to handle a periodic Green's kernel.
%In particular, we present doubly-periodic spectrally-accurate Laplace and Stokes
%two dimensional (2D) fast solvers which handle
%large numbers of smooth inclusions per unit cell.
%Fig.~\ref{f:geom}(a) sketches such a problem in the case of a single
%inclusion $\Omega$ per unit cell.

%where the change from Laplace to Stokes requires only the substitution of
%free-space kernels and consideration of blah
%which also exploit close-evaluation quadratures.

The accurate solution of periodic BVPs has remained a challenging problem
at the forefront of analytical and computational progress
for well over a century \cite{latticesums}.
%going back at least to ionic crystals 
Modern simulations may demand large numbers of objects %inclusions
per unit cell, with arbitrary geometries, that
at high volume fractions may approach arbitrarily close to each other
\cite{sanganimo}.
% LONG LANGE INTERACTIONS?
In such regimes asymptotic methods based upon expansion
in the inclusion size do not apply \cite{drummond,Krop04}. 
Polydisperse suspensions or nonsmooth geometries may require
spatially adaptive discretizations. % for accuracy.
In microfluidics, high aspect ratio and/or skew unit cells are needed,
for instance in the optimal design of particle sorters.
Furthermore,
to obtain accurate average material properties for random media
or suspensions via the Monte Carlo method, thousands of
simulation runs may be needed \cite{helsingzeta2}.
A variety of numerical methods are used (see \cite[Sec.~2.8]{miltonbook}),
including particular solutions
(starting in 1892 with Rayleigh's method for cylinders and spheres
\cite{rayleighsums}, and, more recently, in Stokes flow \cite{sanganiacrivos}),
eigenfunction expansions \cite{wangfibers},
lattice Boltzmann and finite differencing \cite{Ladd94},
and finite element methods \cite{guedes,elastohomog}.
However, for general geometries,
%once geometries move away from simple cylinders and spheres,
it becomes very hard to achieve 
high-order accuracy with any of the above
apart from finite elements, and the
cost of meshing renders the latter unattractive when moving
geometries are involved.
Integral equation methods \cite{Ladyzhenskaya,pozrikidis,LIE,Liu09book}
are natural,
since the PDE has piecewise-constant coefficients, and are very popular
\cite{larsonhigdon2,helsingelasto,Moura94,Krop04,Otani06,cazeaux,klint,mobility}.
By using potential theory to represent the solution field
in terms of the convolution of a free-space Green's function with an
unknown ``density'' function $\tau$ living only on
%inclusion
material boundaries, %$\pO$,
one vastly reduces the number of discretized unknowns.
The %(non-periodic)
linear system resulting by applying boundary conditions
%on $\pO$
takes the form
\be
A\tau \; =\; f~,
\label{A}
\ee
where the $N\times N$ matrix $A$ is the discretization of an integral operator.
The latter is often of Fredholm second-kind, hence $A$ remains well-conditioned,
independent of $N$, so that iterative methods converge rapidly.
Further advantages include their high-order or spectral accuracy
(e.g.\ via Nystr\"om quadratures),
the ease of adaptivity on the boundary, and the existence of fast algorithms
to apply $A$ with optimal $\bigO(N)$ cost, such as the fast multipole method
(FMM) \cite{lapFMM}.

\afterpage{%%%%%%%%%%%%%%%%%%%%%%%%%%%%%%%%%%%%%%%%%%%%%%%%%%%%
\bfi[p]
\bc\ig{width=5.5in}{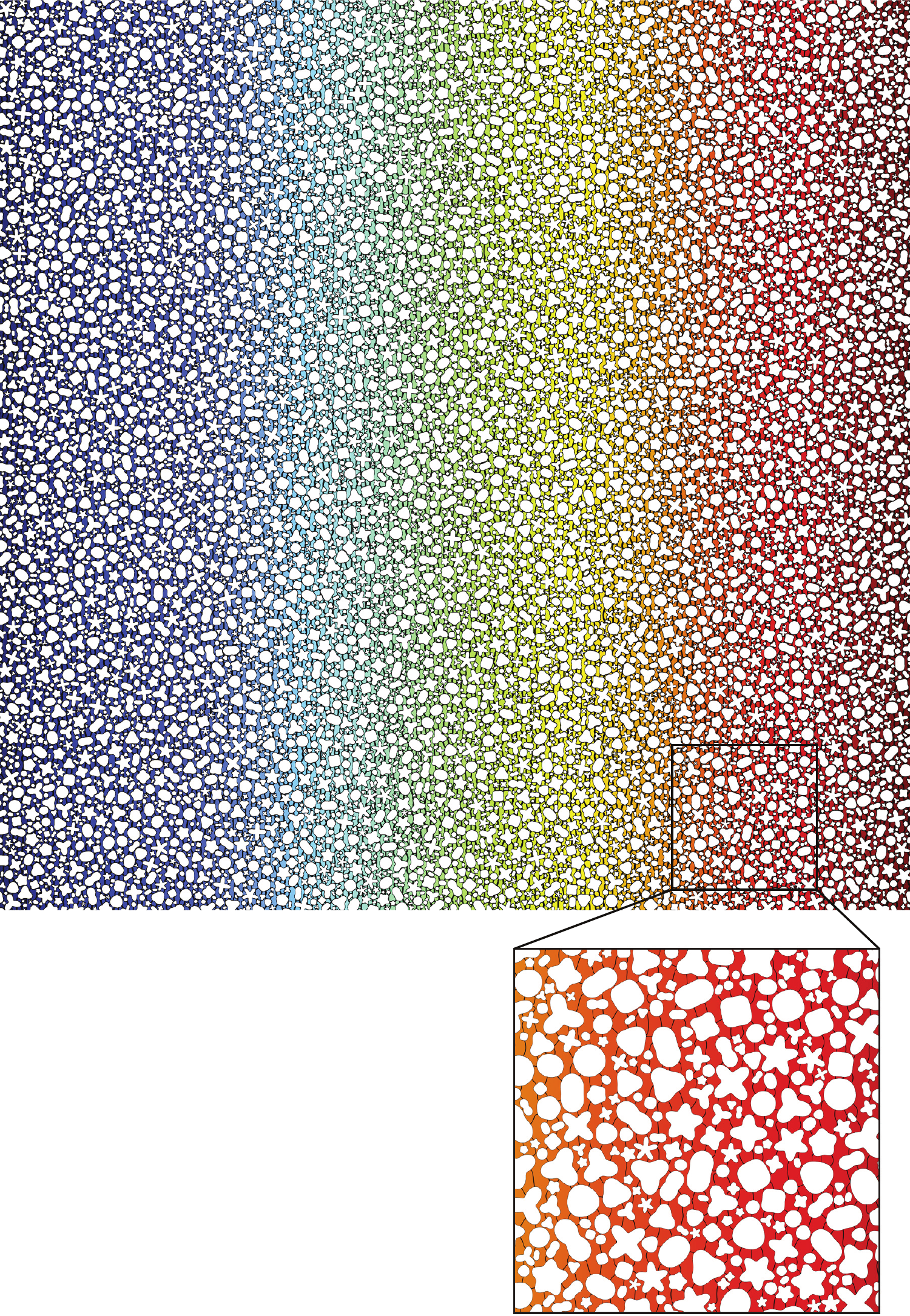}\ec
\ca{Solution of the Laplace equation in a doubly-periodic domain with Neumann boundary conditions on each of $K=10^4$ inclusions, driven by
specified potential drop $\mbf{p}=(1,0)$ (Example 2).
A single unit cell is shown, with the solution
$u$ indicated by a color scale and contours.
The inset shows geometric detail.
Each inclusion has $N_k=512$ discretization nodes on its boundary, resulting in 5.12 million degrees of freedom. The problem took about 28 hours
to solve using two eight-core 2.60 GHz Intel Xeon E5-2670 processors and 64 GB of RAM.
%*** ERROR ?
}{f:lapcandy}
\efi
\clearpage
}

\afterpage{%%%%%%%%%%%%%%%%%%%%%%%%%%%%%%%%%%%%%%%%%%%%%%%%
\bfi[p]
\bc\ig{width=6.3in}{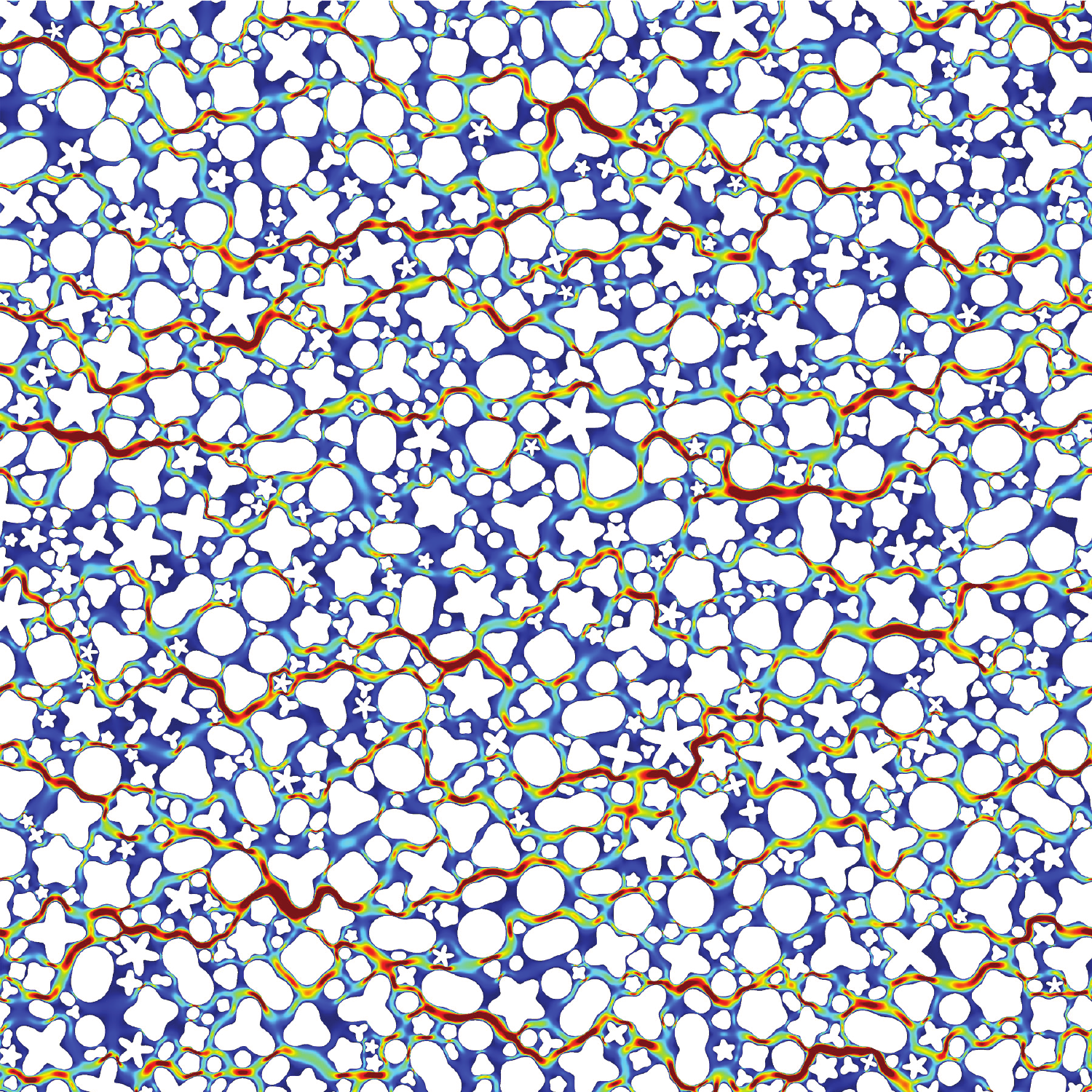}\ec
\ca{Periodic Stokes flow using in doubly-periodic domain with no-slip
boundary conditions on each of $K=10^3$ inclusions (Example 6).
A single unit cell is shown.
We apply no-slip boundary condition at the inclusion boundaries and a pressure difference which drives the flow from left to right. Color here indicates the fluid speed (red is high and blue is low). Each inclusion boundary has $N_k=350$ discretization points, resulting in 700000 degrees of freedom (not accounting for copies). The estimated error in the computed flux is $9 \times 10^{-9}$.
%(reference solution is computed using a finer discretization).
%*** TIME? 
}{f:stocandy}
\efi
\clearpage
}

The traditional integral equation approach to periodic problems
replaces the free-space Green's function by a {\em periodic Green's function},
so that the density is solved on only the geometry lying in a single
unit cell \cite{helsingelasto,Moura94,Krop04,Otani06,cazeaux}.
There is an extensive literature on the evaluation of
such periodic Green's functions (e.g. in the Stokes case see \cite{pozper,vandervorst});
%especially in the Helmholtz case \cite{kurkcu}
yet any such pointwise evaluation for each source-target
pair leads to $O(N^2)$ complexity, which is unacceptable for large problems.
%The challenge is then the fast application of this periodic Green's kernel
%between all sources and targets on these boundaries.
%To recover a fast algorithm, one needs a way to evaluate at the $N$ targets
%the potential due to the $N$ periodized sources together.
There are two popular approaches to addressing this in a way compatible
with fast algorithms:
%lattice sums and particle-mesh Ewald methods.
%There are several popular fast methods for applying the kernel.
\bi
\item
Lattice sums.
Noticing that the difference between the periodic and free-space
Green's function is a smooth PDE solution in the unit cell,
one expands this in a particular solution basis
(a cylindrical or spherical expansion); the resulting
coefficients, which need to be computed only once for a given unit cell,
are called {\em lattice sums}
\cite{fmmlattice,helsingdiscs,Moura94,Krop04}.
% \cite{Otani06}? can't find
They originate in the work of Rayleigh \cite{rayleighsums}
and in the study of ionic crystals
(both reviewed in \cite[Ch.~2--3]{latticesums}).
The FMM may then be periodized by combining the
top-level multipole expansion coefficients with the lattice sums
to give a correction to the local expansion coefficients
(in 2D this is a discrete convolution)
which may be applied fast \cite[Sec.~4.1]{lapFMM} \cite{Otani06}.
This method has been used for the doubly-periodic Laplace BVP
by Greengard--Moura \cite{Moura94}, and Stokes BVP by Greengard--Kropinski
\cite{Krop04}. %, problems similar to those we study.
%There are several complications.
\item
Particle-mesh Ewald (PME) methods.
These methods exploit Ewald's realization \cite{ewald} that,
although both the spatial and the spectral (Fourier) sums for the periodic
Green's function in general converge slowly,
there is an analytic %``partition of unity''
splitting into
spatial and spectral parts such that both terms converge superalgebraically.
Numerically, the spatial part is now local,
while the spectral part may be applied by ``smearing'' onto a uniform
grid, using a pair of FFTs, and (as in the non-uniform FFT \cite{dutt})
correcting for the smearing.
The result is a $\bigO(N\log N)$ fast algorithm;
for a review see \cite{PMErev}.
Recently this has been improved
by Lindbo--Tornberg to achieve overall spectral accuracy
in the Laplace \cite{Lindbo11} and Stokes \cite{Lindbo10} settings,
with applications to 3D fluid suspensions \cite{klint}.
\ei
%particle-mesh Ewald (not adaptive, not fast when clumpy particle distributions).
%Ewald \cite{arens10}. is for 1-periodic in 2d.

The scheme that we present is based purely on free-space Green's functions,
and has distinct advantages over the above two periodization approaches.
1) Only physically-meaningful boundary conditions and compatibility
conditions are used.
Aside from conceptual and algebraic simplification,
this also removes the need for (often mysterious and subtle \cite{latticesums})
choices of the values of various
conditionally or even non-convergent lattice sums based on physical arguments
\cite{drummond,lapFMM,helsingelasto,Moura94}
(e.g., six such choices are needed in \cite{Krop04}).
When it comes to skewed unit cells combined with applied pressure
driving, such hand-picked choices in lattice sum methods become far from
obvious.
%and the various ``uniform background charge'' approaches to ensuring such sums exist \cite{Hasimoto}.
%making for an easy generalization from Laplace to Stokes;
2)
Free-space FMM codes may be inserted without modification,
and fast {\em direct} solvers (based on
hierarchical compression of matrix inverses)
\cite{hackbusch,m2011_1D_survey} require only slight modification
\cite{qpfds,periodicp}.
3) When multi-scale spatial features are present in the problem---commonly observed in practical applications such as polydisperse suspensions and complex microfludic chip geometries---PME codes become inefficient because of their need for uniform grids. In contrast, our algorithm retains the fast nature of the adaptive FMM
in such cases.
%
%Due to the fixed length scale needed for the Ewald split,
%if many particles lie close to each other then PME
%codes suffer due to the local $\bigO(N^2)$ complexity of evaluating
%local interactions. REF.
%Such situations are common in applications such as 
%polydisperse suspensions and adaptive boundary discretizations. REF.
4)
In contrast to lattice sums and Ewald methods,
which intrinsically rely on radially-symmetric expansions and kernels,
our scheme can handle reasonably high aspect ratio %and skew
unit cells with little penalty. %(see Sec.~\ref{s:conc}).
5)
Since our Stokes formulation
does not rely on complex variables (e.g.\ Sherman--Lauricella \cite{Krop04}),
the whole scheme generalizes to 3D with no conceptual changes \cite{garythesis}.
%Variants of the scheme efficiently solve related periodic problems
%in 2D and 3D \cite{qpsc,qpfds,periodicp,gumerov,acper}. Boring self-promotion.
%
A disadvantage of our scheme is that the
prefactor may be slightly worse than Ewald methods, due to the direct
summation of neighboring images. % (however, see Sec.~\ref{s:conc}).
% does that make sense when we're O(N) but Ewald is N log N ?
%
%3x3 direct sum; this can be alleviated by removing sources outside the proxy circle.
%PDEs (Laplace, Helmholtz, Stokes).

The basic idea is simple, and applies to integral representations for a
variety of elliptic PDEs.
Let $\cb$ be a (generally trapezoidal) unit cell ``box'',
and let $\pO$ be the material boundary which is to be periodized
($\pO$ may even intersect the walls of $\cb$).
Applying periodic boundary conditions is equivalent to summing the free-space
representation on $\pO$ over the infinite lattice, as in Fig.~\ref{f:geom}(a).
Our scheme sums this potential representation over
only the $3\times 3$ nearest neighbor images of $\pO$
(see Fig.~\ref{f:geom}(b)), but adds a small
%We enlarge the set of unknowns to comprise not only
%the density $\tau$ on the material boundaries in a single unit cell, but an
%(around $10^2$ unknowns, independent of $N$)
auxiliary basis of smooth PDE solutions in $\cb$,
with coefficient vector $\xi$, that efficiently
represents the (distant) contribution of the rest of the infinite image lattice.
For this auxiliary basis we use point sources (Fig.~\ref{f:geom}(c));
around $10^2$ are needed, a small number which, 
crucially, is {\em independent of} $N$ and the boundary complexity.
We apply the usual (homogeneous) boundary condition on $\pO$,
which forms the first block row of a linear system.
We impose the desired physical periodicity
as auxiliary conditions on the Cauchy data differences %(``discrepancies'')
between wall pairs $R$--$L$ and $U$--$D$ (see Fig.~\ref{f:geom}(a)),
giving a second block row.
The result is a $2\times 2$ block ``extended linear system'' (ELS),
\be
\mt{A}{B}{C}{Q}\vt{\btau}{\xi} \; = \; \vt{0}{\mbf{g}}
~.
\label{ELS}
\ee
%A non-zero
Here $\mbf{g}$ accounts for the applied macroscopic thermal or pressure
gradient in the form of prescribed jumps across one unit cell.
Fig.~\ref{f:geom}(b)--(c) sketches the interactions described by
the four operators $A$, $B$, $C$, and $Q$.
This linear system is generally {\em rectangular and highly ill-conditioned},
but %since it expresses the desired periodic BVP on a single unit cell,
when solved in a backward-stable least-squares sense can result in
accuracies close to machine precision.

Three main routes to a solution of \eqref{ELS} are clear.
(a) The simplest is to use a dense direct least-squares
solve (e.g., via QR), but the $O(N^3)$ cost
becomes impractical for large problems with $N > 10^4$.
Instead, to create fast algorithms one exploits the fact that the numbers of
auxiliary rows and columns in \eqref{ELS} are both small,
%$O(1)$ (of order $10^2$),
as follows.
(b) One may attempt to eliminate $\xi$ to get the Schur complement
square linear system involving $A_\tbox{per}$, the periodized version of $A$,
of the form
\be
A_\tbox{per} \btau \; := \; (A - B Q^+ C) \btau \; = \; - B Q^+ \mbf{g}~,
\label{Aper}
\ee
where $Q^+$ is a pseudoinverse of $Q$ (see Section~\ref{s:schur}).
%Assuming $A_\tbox{per}$ exists, then
As we will see,
\eqref{Aper} can be well-conditioned when \eqref{A} is, and can
be solved iteratively, by
using the FMM to apply $A\btau$ while
applying the second term in its {\em low-rank} factored form $-B((Q^+C)\btau)$.
(c) One may instead eliminate $\btau$ by forming, then applying, a
compressed representation
of $A^{-1}$ via a fast direct solver;
this has proven useful for the case of fixed periodic geometry with multiple 
boundary data vectors \cite{qpfds,periodicp}.
Both methods (b) and (c) can achieve $O(N)$ complexity.
%and hence are appropriate for large problems.

This paper explores route (b). A key contribution is
overcoming the stumbling block that,
%removing the following stumbling block that arises.
for many standard PDEs and integral representations, $A_\tbox{per}$
as written in \eqref{Aper} does not in fact exist.
Intuitively, this is due to divergence in the sum of
the Green's function over the lattice.
For instance, the sum of $\log(1/r)$ clearly diverges,
and thus the single-layer representation for Laplace's equation,
which we use in Section~\ref{s:lap}, cannot be periodized unless the
integral of density (total charge) vanishes.
%Rather than restrict the solution to zero-mean des
This manifests itself in the 2nd block row of \eqref{ELS}:
the range of $C$ contains vectors that are not in the range of $Q$,
thus $Q^+ C$ numerically blows up.
After studying in Sections \ref{s:lapempty} and \ref{s:stoempty}
the consistency conditions for the linear system involving
$Q$ (the ``empty unit cell'' problem),
we propose rank-1 (for Neumann Laplace) and rank-3 (for no-slip Stokes)
corrections to the ELS that allow
the Schur complement to be taken, and, moreover
that remove the nullspace associated with the physical BVP.
We justify these schemes rigorously, and test them numerically.
%Each PDE and integral representation thus brings {\em consistency conditions} on $\tau$;

Although the idea of imposing periodicity via extra linear conditions
has aided certain electromagnetics solvers for some decades
\cite{hafner90,yuan08},
we believe that this idea was first combined with integral equations
by the first author and Greengard in \cite{qplp},
where $\xi$ controlled a Helmholtz local expansion.
%(Fourier--Bessel)
Since then it has become popular for a variety of PDE
\cite{qplp,qpfds,gumerov,mlqp,cazeaux,periodicp,acper}.
From \cite{qplp} we also inherit the split into near and far images, and
cancellations in $C$ that allow a rapidly convergent scheme
even when $\pO$ intersects unit cell walls.
The use of point sources as a particular solution basis that is efficient
for smooth solutions is known as the
``method of fundamental solutions'' (MFS) \cite{Bo85,FaKa98},
``method of auxiliary sources'' \cite{Kupradze67},
``charge simulation method'' \cite{Ka88},
or, in fast solvers, ``equivalent source'' \cite{bruno01} or
``proxy'' \cite{Mart07} representations.
This is also used in the recent 3D periodization scheme
of Gumerov--Duraiswami \cite{gumerov}.
Finally, the low-rank perturbations which
%remove the left null-space
enlarge the range of $Q$ are inspired by the low-rank perturbation methods
for singular square systems of Sifuentes et al.\ \cite{sifuentes2014randomized}.

Here is a guide to the rest of this paper.
In Section~\ref{s:lap} we present the periodic Neumann Laplace BVP,
possibly the simplest application of the scheme.
This first requires understanding and numerically solving
the ``empty unit cell'' subproblem, which we do in Section~\ref{s:lapempty}.
The integral formulation, discretization,
and direct numerical solution of the
ELS \eqref{ELS} follows in Section~\ref{s:lapbie}.
A general scheme to stably take the Schur complement is
given in Section~\ref{s:schur}, along with a scheme
to remove the physical nullspace specific to our representation.
The latter is tested numerically.
Section~\ref{s:multi} shows how the FMM and close-evaluation quadratures are
incorporated,
and applies them to large problems with thousands of inclusions.
Section~\ref{s:kappa} defines the effective conductivity tensor $\kappa$
and shows how to evaluate it efficiently using a pair of BVP solutions.
In Section~\ref{s:sto} we move to periodic Dirichlet (no-slip) Stokes flow,
and show how its periodizing scheme closely parallels the Laplace version.
In fact, the only differences are
new consistency conditions for the empty subproblem
(Section~\ref{s:stoempty}), the use of a 
%mixed ${\cal D}+{\cal S}$
combined-field formulation (Section~\ref{s:stobie}),
%\cite{hebeker} \cite[p.128]{pozrikidis}.
and the need for a rank-three perturbation for a stable Schur complement
(Section~\ref{s:stosch}).
Experiments on the drag of a square array of discs, and on large-scale flow
problems, are performed in Section~\ref{s:stonum}.
We discuss generalizations and conclude in Section~\ref{s:conc}.

Our aim is to illustrate, with two example BVPs, a unified route to
a well-conditioned periodization compatible with fast algorithms.
We take care to state the {\em consistency conditions} and {\em nullspaces}
of the full and empty problems, this being the main area of
problem-specific variation.
%care is needed and
We believe that the approach
adapts simply to other boundary conditions and PDEs,
once the corresponding consistency conditions and nullspaces are laid out.
%formulation, and since the literature is a little loose.
Although general aspect ratios and skew unit cells are amongst
our motivations,
%we postpone these for a future publication, and
%here
for clarity we stick to the unit square; the generalization is
straightforward.

%For both BVPs we emphasize a unified route to a
%well-conditioned periodization compatible with iterative methods and
%fast algorithms:
%i) understanding the $k$ consistency conditions for the empty unit cell BVP,
%ii) rank-$k$ perturbation of $Q$ in the ELS,
%followed by iii) Schur complement.
%Since we believe that this scheme can apply in a variety of other BVPs
%of engineering importance, our goal is to aid the reader who wishes to
%pursue this.

%Our goal aiding the reader who wishes to apply them to other BVPs.
%We have made an effort to present the periodization methods
%in a language common to both BVPs that we study,

\begin{rmk}[Software]
We maintain documented MATLAB codes implementing the basic
methods in this paper at {\tt http://github.com/ahbarnett/BIE2D}
\label{r:code}\end{rmk}

% LLLLLLLLLLLLLLLLLLLLLLLLLLLLLLLLLLLLLLLLLLLLLLLLLLLLLLLLLLLLLLLLLLLLLLLLL
\section{The Neumann Laplace case}
\label{s:lap}

We now present the heat/electrical conduction problem
in the exterior of a periodic lattice of insulating inclusions
(corresponding to $\sigma_d=0$ in \cite{Moura94}).
For simplicity we first assume a single inclusion $\Omega$ per unit cell.
Let $\ex$ and $\ey$ be vectors defining the lattice in $\RR^2$;
%for simplicity
we work with the unit square so that $\ex = (0,1)$ and
$\ey = (1,0)$.
Let $\Omega_\Lambda := \{ \xx \in \RR^2: \xx + m\ex + n\ey \in \Omega
\mbox{ for some } m,n\in\mathbb{Z}\}$
represent the infinite lattice of inclusions.
The scalar $u$, representing electric potential or temperature,
solves the BVP
\bea
\Delta u & = & 0 \qquad \mbox{ in } \RR^2 \backslash \overline{\Omega_\Lambda}
\label{lap}
\\
u_n & = & 0 \qquad \mbox{ on } \partial
\Omega_\Lambda
\label{bc}
\\
u(\xx+\ex) - u(\xx) &=& p_1 \qquad \mbox{for all } \xx \in \RR^2 \backslash \overline{\Omega_\Lambda}
\\
u(\xx+\ey) - u(\xx) &=& p_2 \qquad \mbox{for all } \xx \in \RR^2 \backslash \overline{\Omega_\Lambda}
~,
\label{p2}
\eea
ie, $u$ is harmonic, has zero flux on inclusion boundaries,
and is periodic up to a given pair of constants $\mbf{p}=(p_1,p_2)$
which encode the external (macroscopic) driving.
We use the abbreviation $u_n = \frac{\partial u}{\partial n} = \nn\cdot \nabla u$,
where $\nn$ is the unit normal on the relevant evaluation curve.
\begin{pro} % pppppppppppppppppppppppppppppppppppppppppppppppppp
For each $(p_1,p_2)$ the solution to \eqref{lap}--\eqref{p2}
is unique up to an additive constant.
\label{p:lapnull}\end{pro}
\begin{proof}
As usual, one considers the homogeneous BVP arising when $u$ is the difference
of two solutions.
Let $\cb$ be any unit cell (tiling domain) containing $\Omega$, then using
Green's first identity,
$$0 = \int_{\cbo} u\Delta u = - \int_{\cbo} |\nabla u|^2 + \int_{\partial\cb} u u_n - \int_\pO uu_n
~.$$
The first boundary term cancels by periodicity, the second by \eqref{bc},
hence $\nabla u \equiv 0$.
\end{proof}

To solve the periodic BVP we first re-express it as a BVP on a
single unit cell $\cb$ with coupled boundary values
on the four walls comprising its boundary
$\partial\cb := L \cup R \cup D \cup U$; see Fig.~\ref{f:geom}(a).
For simplicity we assume for now that the square $\cb$ can be chosen such
that $\Omega$ does not intersect any of the walls; this restriction
will later be lifted (see Remark~\ref{r:intersect}).
We use the notation $u_L$ to mean the restriction of $u$ to the wall $L$,
and $u_{nL}$ for its normal derivative using the normal on $L$ (note that
this points to the right, as shown in the figure.)
Consider the following reformulated BVP:
\bea
\Delta u & = & 0 \qquad \mbox{ in } \cb \backslash \overline{\Omega}
\label{lap2}
\\
u_n & = & 0 \qquad \mbox{ on } \partial \Omega
\label{bc2}
\\
u_R - u_L &=& p_1
\label{u1}
\\
u_{nR} - u_{nL} &=& 0
\label{j1}
\\
u_U - u_D &=& p_2
\label{u2}
\\
u_{nU} - u_{nD} &=& 0
~.
\label{j2}
\eea
Clearly any solution to \eqref{lap}--\eqref{p2} also satisfies
\eqref{lap2}--\eqref{j2}.
Because of the unique continuation of Cauchy data $(u,u_n)$ as a solution
to the 2nd-order PDE, the converse holds, thus the two BVPs are
equivalent.
We define the {\em discrepancy} \cite{qplp} of a solution $u$
as the stack of the four functions on the left-hand side of \eqref{u1}--\eqref{j2}.

% eeeeeeeeeeeeeeeeeeeeeeeeeeeeeeeeeeeeeeeeeeeeeeeeeeeeeeeeeeeeeeeeeeeeeeeeeee
\subsection{The empty unit cell discrepancy BVP and its numerical solution}
\label{s:lapempty}

We first analyze, then solve numerically, an important subproblem
which we call the ``empty unit cell BVP.''
We seek a harmonic function $v$ matching a
given discrepancy $g = [g_1;g_2;g_3;g_4]$
(i.e., a stack of four functions defined on the
walls $L$, $L$, $D$, $D$, respectively).
That is,
\bea
\Delta v & = & 0 \qquad \mbox{ in } \cb
\label{lapv} \\
v_R - v_L &=& g_1
\label{g1} \\
v_{nR} - v_{nL} &=& g_2
\label{g2} \\
v_U - v_D &=& g_3
\label{g3} \\
v_{nU} - v_{nD} &=& g_4
\label{g4}
~.
\eea
We now give the consistency condition and nullspace for this BVP,
which shows that it behaves essentially
like a square linear system with nullity 1.
\begin{pro} % pppppppppppppppppppppppppppppppp
A solution $v$ to \eqref{lapv}--\eqref{g4}
exists if and only if $\int_L g_2 ds + \int_D g_4 ds = 0$,
and is then unique up to a constant.
\label{p:lapemptynull}\end{pro}
\begin{proof}
The zero-flux condition $\int_\pcb v_n = 0$ holds for harmonic functions.
Writing $\pcb$ as the union of the four walls, with their normal senses,
gives the sum of the integrals of \eqref{g2} and \eqref{g4}.
Uniqueness follows from the method of proof of Prop.~\ref{p:lapnull}.
\end{proof}

We now describe a numerical solution method for this BVP that is accurate for a
certain class of data $g$, namely those for which the solution $v$
may be continued as a regular solution to Laplace's equation into a large
neighborhood of $\cb$, and hence each of $g_1,\dots,g_4$ is analytic.
It will turn out that this class is sufficient
for our periodic scheme (essentially because $v$ will only have to
represent distant image contributions, as shown in Fig.~\ref{f:geom}(c)).
Let $\cb$ be centered at the origin.
Recalling the fundamental solution to Laplace's equation,
\be
G(\xx,\yy) = \frac{1}{2\pi}\log \frac{1}{r}~,
\qquad r:= \|\xx-\yy\|~,
\label{G}
\ee
we approximate the solution in $\cb$ by a linear combination of
such solutions with source points $\yy_j$ lying uniformly on a circle
of radius $\rp>r_\cb$, where $r_\cb := 1/\sqrt{2}$
is the maximum radius of the unit cell.
That is, for $\xx\in \cb$, 
%which solve \eqref{lapv} 
\be
v(\xx) \; \approx \; \sum_{j=1}^M \xi_j \phi_j(\xx)~,
\qquad \phi_j(\xx) := G(\xx,\yy_j)~,
\qquad \yy_j := (\rp \cos 2\pi j/M, \rp \sin 2\pi j/M)~,
\label{mfs}
\ee
with unknown coefficient vector $\xi := \{ \xi_j\}_{j=1}^M$.
As discussed in the introduction, this is known as the MFS.
Each basis function $\phi_j$ is a particular solution to \eqref{lapv} in $\cb$,
hence only the boundary conditions need enforcing
(as in Rayleigh's original method \cite{rayleighsums}).

This MFS representation is complete for harmonic functions. More precisely,
for $v$ in a suitable class, it is capable of exponential accuracy uniformly
in $\cb$. In particular, since $\cb$ is contained within the ball
$\|\xx\|\le r_\cb$, we may apply known convergence results to get the following.

\begin{thm}  % tttttttttttttttttttttttttttttttttttt
Let $v$ extend as a regular harmonic function throughout the closed ball
$\|\xx\|\le \rho$ of radius $\rho>r_\cb$.
Let the fixed proxy radius $\rp\neq 1$ satisfy $\sqrt{r_\cb \rho} < \rp < \rho$.
For each $M\ge 1$ let $\{\phi^{(M)}_j\}_{j=1}^M$ be a proxy basis set
as in \eqref{mfs}.
Then there is a sequence of coefficient vectors $\xi^{(M)}$, one vector
for each $M$, and a constant $C$ dependent only on $v$, such that
$$
\bigl\|v - \sum_{j=1}^M \xi_j^{(M)} \phi_j^{(M)} \bigr\|_{L^\infty(\cb)}
\; \le \;
C \left(\frac{\rho}{r_\cb}\right)^{-M/2} ~, \qquad M = 1,2,\ldots~.
$$
%holds for all $M\ge 1$.
In addition, the vectors may be chosen so that
the sequence $\|\xi^{(M)}\|_2$, $M=1,2,\ldots$, is bounded.
\label{t:Qconv}
\end{thm}

%\begin{rmk}
This exponential convergence---with rate controlled by the distance to the
nearest singularity in $v$---was derived by Katsurada \cite[Rmk.~2.2]{Ka89} in the context of
collocation (point matching) for a Dirichlet BVP on the boundary of the disc
of radius $r_\cb$, which is enough to guarantee the existence of
such a coefficient vector sequence.
The boundedness of the coefficient norms has been known in the
context of Helmholtz scattering in the Russian literature for
some time (see \cite[Sec.~3.4]{Ky96} and references within, and
\cite[Thm.~2.4]{doicu}).
We do not know of a reference stating this for the Laplace case,
but note that the proof is identical to that in
\cite[Thm.~6]{mfs}, by using equation (13) from that paper.
The restriction $\rp<\rho$ is crucial for high accuracy, since
when $\rp>\rho$, although convergence occurs, $\|\xi^{(M)}\|_2$ blows
up, causing accuracy loss due to catastrophic cancellation \cite[Thm.~7]{mfs}.
%\end{rmk}
\begin{rmk}   % rrrrrrrrrrrrr
Intuitively, the restriction $\rp\neq1$ arises because the proxy basis may
be viewed as a periodic trapezoid rule quadrature approximation to a
single-layer potential on the circle $\|\xx\|=\rp$, a representation which 
is incomplete when $\rp=1$: it cannot represent the constant function 
\cite[Rmk.~1]{YanSloan}.
\end{rmk}

To enforce boundary conditions, let $\xx_{iL} \in L$, $\xx_{iD}\in D$,
$i=1,\dots,m$, be two sets of $m$
collocation points, on the left and bottom wall
respectively; we use %In practice we use %around $m=20$
Gauss--Legendre nodes. % on each wall;
%the resulting accuracy is relatively insensitive to the
%point distribution as long as it doesn't contain large gaps.
Enforcing \eqref{g1} between collocation points on the left and right walls
then gives
\be
\sum_{j=1}^M [\phi_j(\xx_{iL}+\ex) - \phi_j(\xx_{iL})] \xi_j \; = \; g_1(\xx_{iL})
~,
\qquad \mbox{ for all } \; i=1,\dots,m~.
\label{Q1}
\ee
Continuing in this way, the
full set of discrepancy conditions \eqref{g1}--\eqref{g4}
gives the linear system
\be
Q \xi \; = \; \mbf{g}
~,
\label{Qsys}
\ee
where $\mbf{g} \in \RR^{4m}$
%[\{g_1(\xx_{iL})\}_{i=1}^m; \{g_2(\xx_{iL})\}_{i=1}^m;\{g_3(\xx_{iD})\}_{i=1}^m; \{g_4(\xx_{iD})\}_{i=1}^m]
stacks the four vectors of discrepancies sampled at the collocation points,
while the matrix $Q = [Q_1;Q_2;Q_3;Q_4]$ consists of four block rows,
each of size $m\times M$.
By reading from \eqref{Q1}, one sees that $Q_1$ has elements
$(Q_1)_{ij} = \phi_j(\xx_{iL}+\ex) - \phi_j(\xx_{iL})$.
Analogously,
$(Q_2)_{ij} = \frac{\partial\phi_j}{\partial n}(\xx_{iL}+\ex) -
\frac{\partial\phi_j}{\partial n}(\xx_{iL})$,
$(Q_3)_{ij} = \phi_j(\xx_{iD}+\ey) - \phi_j(\xx_{iD})$,
and finally
$(Q_4)_{ij} = \frac{\partial\phi_j}{\partial n}(\xx_{iD}+\ey) -
\frac{\partial\phi_j}{\partial n}(\xx_{iD})$.

\bfi  % fffffffffffffffffffffffffffffffffffffffffffffffffffffffffffffff
\ig{width=2in}{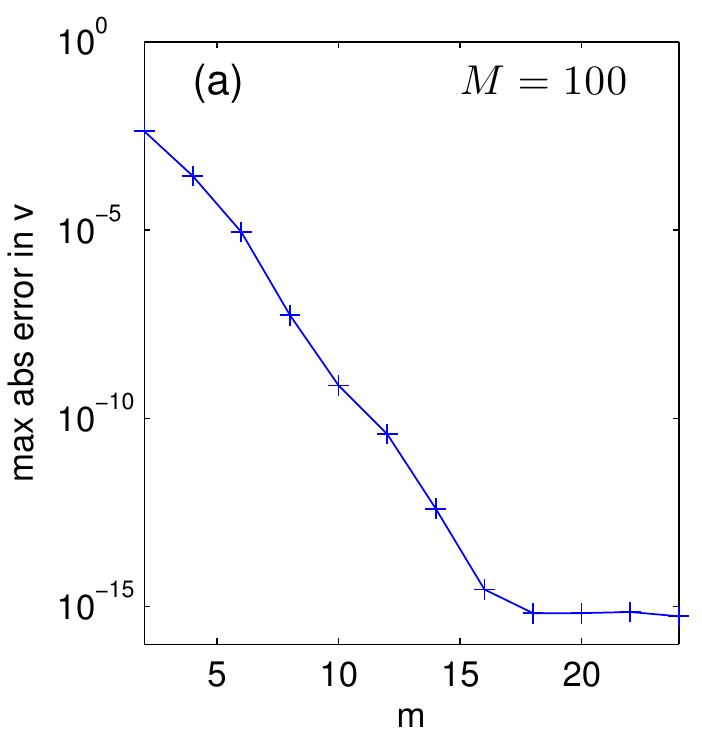}
\quad
\ig{width=2in}{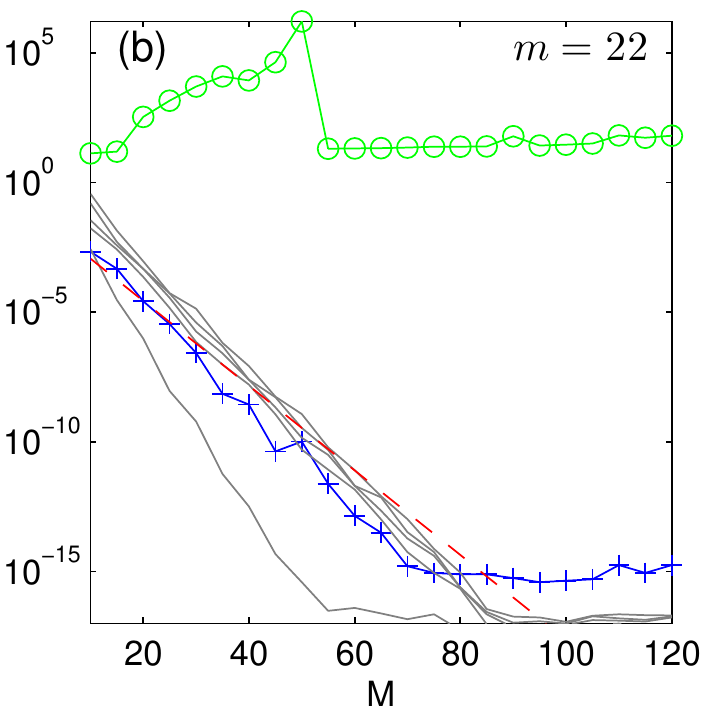}
\quad
\ig{width=2in}{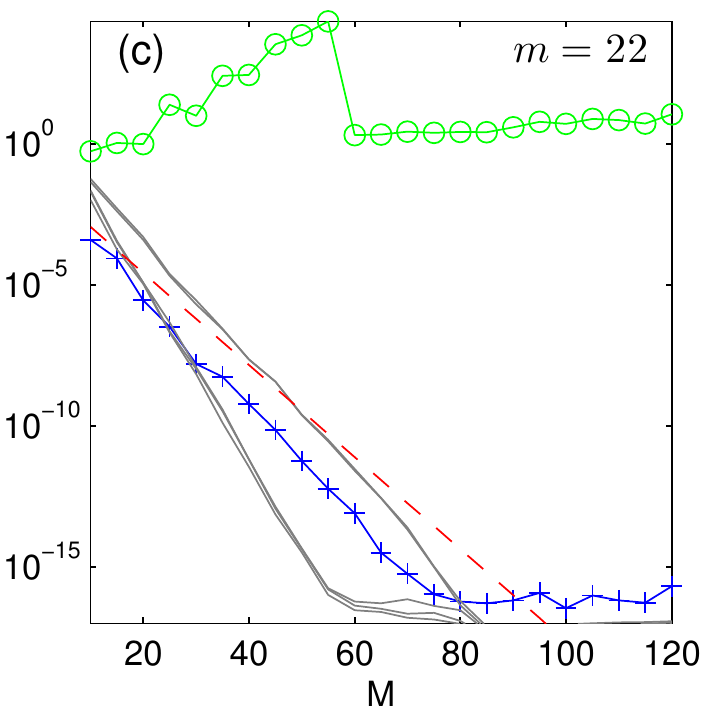}
\ca{Convergence of the proxy point numerical
scheme for the empty discrepancy BVP, for the case of
known solution with singularity $\xx_0 = 1.5(\cos .7,\sin .7)$ is a
distance $1.5$ from the center of the square unit cell of side $1$.
The proxy radius is $\rp=1.4$.
(a)-(b) are for the Laplace case with $g$ deriving from the known solution
$v(\xx) = \log \| \xx - \xx_0\|$; (c) is for the Stokes case
with known $\mbf{v}$ coming from a stokeslet at $\xx_0$ with force
$\ff_0 = [0.3;-0.6]$.
(a) Shows convergence in maximum absolute error in $v$ at 100 target points
interior to $\cb$, vs the number $m$ of wall quadrature nodes,
with fixed $M=100$ proxy points.
(b) Shows convergence in $M$ with fixed $m=22$ ($+$ signs)
and its prediction via Thm.~\ref{t:Qconv} (dotted line),
the lowest five singular values of $Q$ (grey lines),
and the solution vector norm $\|\xi\|_2$ (circles).
(c) is same as (b) but for Stokes.
Note that in the Laplace case (b) there is one singular value smaller
than the others, whereas for Stokes (c) there are three such singular
values.
}{f:Qconv}
\efi   % fffffffffffffffffffffffffff

We solve the (generally rectangular)
system \eqref{Qsys} in the least squares sense,
which corresponds to minimizing a weighted $L^2$-norm of the discrepancy error
(we find in practice that there is no advantage to
incorporating the square-root of the weights associated with the
collocation nodes).
As is well known in the MFS community \cite{Ka89,FaKa98}, $Q$
becomes exponentially ill-conditioned
%(and therefore rank-deficient)
as $M$ grows;
intuitively this follows from the exponential decay of
the singular values of the single-layer operator from the proxy radius $\rp$
to the circle radius $r_\cb$ containing $\cb$.
Thus, a least squares solver is needed which can handle rank-deficiency stably,
i.e.\ return a small-norm solution when one exists.
In this case the ill-conditioning causes no loss of accuracy, and,
even though there is instability in the vector $\xi$,
for the evaluation of $v$ errors close to machine precision ($\emach$) are % easily
achieved \cite{mfs}.
\begin{rmk}
For this and subsequent direct solves
we use MATLAB's {\tt linsolve} (with the option {\tt RECT=true}
to prevent the less accurate LU from being used in the square case),
% not backslash ({\tt mldivide}) any more!
which uses column-pivoted QR to find the so-called
``basic'' solution \cite[Sec.~5.7]{molermatlab} having at most $r$ nonzero
entries, where $r$ is the numerical rank;
this is close to having minimum norm \cite[Sec.~5.5]{golubvanloan}.
\label{r:linsolve}
\end{rmk}

We illustrate this %Fig.~\ref{f:Qconv}
with a simple numerical test, in which
$g$ is the discrepancy of a known harmonic function $v$ of typical magnitude $\bigO(1)$ and with sufficiently
distant singularity.
Once $Q$ is filled and \eqref{Qsys} solved, the numerical solution
is evaluated via \eqref{mfs}, and the maximum error at 100
random target points in $\cb$ is taken (after removing an
overall constant error, expected from Prop.~\ref{p:lapemptynull}).
Fig.~\ref{f:Qconv}(a)
shows exponential convergence in this error vs the number of wall nodes $m$.
Fig.~\ref{f:Qconv}(b) shows exponential convergence
with respect to $M$, the number of proxy points, with
a rate slightly exceeding the predicted rate.
It is clear that
whenever $m\ge 20$ and $M\ge 70$ the norm $\|\xi\|_2$ remains $\bigO(1)$,
and around 15 digits of accuracy result.

The decaying lowest few singular values of $Q$ are also shown in panel (b):
in particular
there is {\em one} singular value decaying faster than all others.
It is easy to verify that this
corresponds to the null-space of the BVP (Prop.~\ref{p:lapemptynull});
indeed we have tested that
its right singular vector is approximately constant, and generates via
\eqref{mfs} the constant function in $\cb$
to within $\bigO(\emach)$ relative error.
Likewise, the consistency condition in Prop.~\ref{p:lapemptynull}
manifests itself in $\Nul Q\tr$:
let $\mbf{w}$ be the vector that applies the discretization of
this consistency condition to a vector $\mbf{g}$, namely
\be
\mbf{w} \; := \; [\mbf{0}_m;\mbf{w}_L;\mbf{0}_m;\mbf{w}_D] \; \in \; \RR^{4m}
~,
\label{w}
\ee
where semicolons indicate vertical stacking,
$\mbf{w}_L,\mbf{w}_D \in \RR^m$ are the vectors of weights
corresponding to the collocation nodes on $L$, $D$,
and $\mbf{0}_m\in\RR^m$ is the zero vector. %with $m$ components.
Then we expect that
\be
\mbf{w}\tr Q \; \approx \; \mbf{0}_M\tr
~,
\label{wQ}
\ee
and indeed observe numerically that $\|\mbf{w}\tr Q\|_2 = \bigO(\emach)$
once $m\ge 20$.
In summary, although the matrix $Q$
is generally rectangular and ill-conditioned, it also
inherits both aspects of the unit nullity of the empty BVP that it discretizes.
%In what sense does the linear system inherit the nullity 1 of the
%underlying empty BVP?

\begin{rmk}
A different scheme is possible in which $Q$ would be square, and
(modulo a nullity of one, as above)
well-conditioned, based on a ``tic-tac-toe'' set of layer potentials
(see \cite[Sec.~4.2]{qplp} in the Helmholtz case).
However, we recommend the above proxy point version, since
i) the matrix $Q$ is so small that handling its ill-conditioning is very cheap,
ii) the tic-tac-toe scheme
demands close-evaluation quadratures for its layer potentials,
and iii) the tic-tac-toe scheme is more complicated.
\end{rmk}

% iiiiiiiiiiiiiiiiiiiiiiiiiiiiiiiiiiiiiiiiiiiiiiiiiiiiiiiiiiiiiiiiiiiiiiii
\subsection{Extended linear system for the conduction problem}
\label{s:lapbie}

We now treat the above empty BVP solution scheme as a component in a
scheme for the periodic BVP \eqref{lap2}--\eqref{j2}.
Simply put, we
take standard potential theory for the Laplace equation \cite[Ch.~6]{LIE},
and augment this by enforcing periodic boundary conditions.
Given a density function $\tau$ on the inclusion boundary $\pO$,
recall that the single-layer potential, evaluated at a target point $\xx$,
is defined by  %  in R2 \ omega_lambda ?
\be
v = ({\cal S}_\pO \tau)(\xx) := \int_\pO G(\xx,\yy) \tau(\yy) ds_\yy
= \frac{1}{2\pi}\int_\pO \log \frac{1}{\|\xx-\yy\|} \tau(\yy) ds_\yy
~,\qquad
\xx\in\RR^2~.
\label{lapS}
\ee
%where the Laplace kernel $G$ is \eqref{G}.
Using $\nx$ to indicate the outward normal at $\xx$,
this potential obeys the jump relation \cite[Thm.~6.18]{LIE}
\be
v_n^\pm :=
\lim_{h\to0^+} \nx \cdot \nabla ({\cal S}_\pO \tau)(\xx \pm h\nx)
\; = \;
\bigl((\mp\half + D\tr_{\pO,\pO})\tau\bigr)(\xx)
~,
\label{JR}
\ee
where $D\tr_{\Gamma',\Gamma}$ denotes the usual transposed
double-layer operator from a general source curve $\Gamma$
to a target curve $\Gamma'$, defined by
\be
(D\tr_{\Gamma',\Gamma}\tau)(\xx) = \int_\Gamma \frac{\partial G(\xx,\yy)}{\partial \nx} \tau(\yy) ds_\yy
 = \frac{1}{2\pi} \int_\Gamma \frac{-(\xx-\yy)\cdot \nx}{\|\xx-\yy\|^2} \tau(\yy) ds_\yy
~, \qquad \xx \in \Gamma'~.
\label{DT}
\ee
The integral implied by the self-interaction operator $D\tr_{\pO,\pO}$ is to be interpreted in the principal value sense.

Our representation for the solution
sums the single-layer potential over the $3\times 3$ copies closest
to the origin, and adds an auxiliary basis as in \eqref{mfs},
\be
u \; = \; {\cal S}^\tbox{near}_\pO\tau + \sum_{j=1}^M \xi_j \phi_j
~,
\qquad
\mbox{where }\;
({\cal S}^\tbox{near}_\pO\tau)(\xx)
%\;\; := \!\! \sum_{m,n\in\{-1,0,1\}} {\cal S}_{\pO+m\ex+n\ey} \tau
\;\; := \!\!
\sum_{m,n\in\{-1,0,1\}} \int_\pO G(\xx,\yy+m\ex+n\ey) \tau(\yy) ds_\yy
~.
\label{urep}
\ee
Our unknowns are the density $\tau$ and auxiliary vector $\xi$.
Substituting \eqref{urep} into the Neumann boundary condition \eqref{bc2},
and using the exterior jump relation on the central copy ($m=n=0$) only,
gives our first block row,
\be
(-\half + D^{\tbox{near},T}_{\pO,\pO})\tau + \sum_{j=1}^M \xi_j \phi_j|_\pO
\;= \;0~,
\label{bie}
\ee
where, as before, the ``near'' superscript
denotes summation over source images as in \eqref{urep}.
%Note that the half identity term arises from the jump relation for
%only the $m=n=0$ central copy.

The second block row arises as follows.
Consider the substitution of \eqref{urep}
into the first discrepancy equation \eqref{u1}: there are
nine source copies, each of which interacts with the $L$ and $R$ walls,
giving 18 terms. However, the effect of the rightmost six
sources on $R$ is cancelled by the effect of the leftmost six sources on
$L$, leaving only six terms, as in \cite[Fig.~4(a)--(b)]{qplp}.
All of these surviving terms involve {\em distant}
interactions (the distances exceed one period if $\Omega$ is contained
in $\cb$).
Similar cancellations occur in the remaining three equations
\eqref{j1}--\eqref{j2}.
The resulting four subblocks are
\bea
\sum_{n\in\{-1,0,1\}} ( S_{R,\pO-\ex+n\ey} - S_{L,\pO+\ex+n\ey} )\tau
+ \sum_{j=1}^M (\phi_j|_R-\phi_j|_L) \xi_j
&=& p_1
\label{row2a}
\\
\sum_{n\in\{-1,0,1\}} ( D\tr_{R,\pO-\ex+n\ey} - D\tr_{L,\pO+\ex+n\ey} )\tau
+ \sum_{j=1}^M \biggl(\bigl.\frac{\partial\phi_j}{\partial n}\bigr|_R-\bigl.\frac{\partial\phi_j}{\partial n}\bigr|_L\biggr) \xi_j
&=& 0
\label{row2b}
\\
\sum_{m\in\{-1,0,1\}} ( S_{U,\pO+m\ex-\ey} - S_{D,\pO+m\ex+\ey} )\tau
+ \sum_{j=1}^M (\phi_j|_U-\phi_j|_D) \xi_j
&=& p_2
\label{row2c}
\\
\sum_{m\in\{-1,0,1\}} ( D\tr_{U,\pO+m\ex-\ey} - D\tr_{D,\pO+m\ex+\ey} )\tau
+ \sum_{j=1}^M \biggl(\bigl.\frac{\partial\phi_j}{\partial n}\bigr|_U-\bigl.\frac{\partial\phi_j}{\partial n}\bigr|_D\biggr) \xi_j
&=& 0
~.
\label{row2d}
\eea
\begin{rmk}[Wall intersection] \label{r:intersect}
The cancellation of all near interactions in \eqref{row2a}--\eqref{row2d}
is due to the $3\times 3$ neighbor summation in the representation \eqref{urep}.
%involving summation over the $3\times 3$ self and nearest neighbors.
Furthermore this cancellation allows an accurate solution even when
$\Omega$ intersects $\pcb$,
without the need for specialized quadratures,
as long as $\Omega$
remains far from the boundary of the enclosing $3\times 3$ unit cell block.
Informally, the unit cell walls are ``invisible'' to the inclusions.
With an elongated $\Omega$ for which the
last condition does not hold, with a little bookkeeping
one could split $\pO$ into pieces that, after lattice translations,
lie far from the boundary of the enclosing $3\times 3$ unit cell block; we
leave this last case for future work.
\end{rmk}

\eqref{bie}--\eqref{row2d} form a set of coupled integral-algebraic equations,
where the only discrete aspect is that of the $\bigO(1)$ proxy points.%
\footnote{It would also be possible to write a purely continuous version
by replacing the proxy circle by a single-layer potential; however,
sometimes an intrinsically discrete basis $\{\phi_j\}$ is useful
\cite{qplp,mlqp}.}
It is natural to ask how the unit nullity of the BVP manifests itself in the
solution space for the pair $(\tau,\xi)$. Are there
pairs $(\tau,\xi)$ with no effect on $u$, enlarging the nullspace?
It turns out that the answer is no, and that the one-dimensional
nullspace (constant functions)
is spanned purely by $\xi$, as a little potential theory now shows.
% This is due to SLP being unique rep for Ext Neu Lap BVP *** SAY?
%
\begin{lem}  % lllllllllllllllllllllllllllllllllllllllllllll
In the solution to \eqref{bie}--\eqref{row2d},
$\tau$ is unique.
%and $\sum_{j=1}^M \xi_j \phi_j$ is unique up to a constant function.
\label{l:laptau}\end{lem}
\begin{proof}
Let $(\tau,\xi)$ be the difference between any two solutions to
\eqref{bie}--\eqref{row2d}.
Let $v$ be the representation \eqref{urep} using this $(\tau,\xi)$,
both in $\cbo$ but also inside $\Omega$.
Then by construction $v$ is a solution to the homogeneous
BVP \eqref{lap2}--\eqref{j2}, i.e.\ with $p_1=p_2=0$, thus by
Prop.~\ref{p:lapnull}, $v$ is constant in $\cbo$.
Thus by the continuity of the single-layer potential \cite[Theorem~6.18]{LIE},
the interior limit of $v$ on $\pO$ is constant.
However, $v$ is harmonic in $\Omega$, so $v$ is constant in $\Omega$.
By \eqref{JR}, $v_n^+ - v_n^- = -\tau$, but we have just shown that
both $v_n^+$ and $v_n^-$ vanish, so $\tau\equiv 0$.
%Then, by \eqref{urep}, $\sum_{j=1}^M \xi_j \phi_j$ is constant.
\end{proof}
%Note that the middle part of this proof is the same as that
%for the uniqueness of the single-layer representation for the exterior
%Laplace Neumann BVP \cite[]{}
% NO - I DON'T LIKE KRESS VERSION OF EXT NEU BVP - I prefer uniqueness if
% condition on data g is removed, and instead u- int(g).log(x) = o(1).

\subsubsection{Discretization of the ELS}

We discretize \eqref{bie},
the first %(inclusion boundary integral)
row of the system,
using a set of quadrature nodes $\{\xx_i\}_{i=1}^N$ on $\pO$
and weights $\{w_i\}_{i=1}^N$ such that
$$
\int_\pO f(\yy) ds_\yy \; \approx \; \sum_{i=1}^N w_i f(\xx_i)
$$
holds to high accuracy for smooth functions $f$.
In practice, when $\pO$ is parametrized by a $2\pi$-periodic function
$\xx(t)$, $0\le t<2\pi$, then using the periodic trapezoid rule in $t$
gives $\xx_i=\xx(2\pi i/N)$ and $w_i = (2\pi/N)\|\xx'(2\pi i/N)\|$.
Since $D\tr$ has a smooth kernel, we apply
Nystr\"om discretization \cite[Sec.~12.2]{LIE} to the integral
equation \eqref{bie} using these nodes, to get our first block row
\be
A \btau + B \xi \; = \; 0~,
\label{row1}
\ee
where $\btau = \{\tau_i\}_{i=1}^N$
is a vector of density values at the nodes, and
$A\in\RR^{N\times N}$ has entries
\be
A_{ij} \; = \;
-\half \delta_{ij} \; + \!\!
\sum_{m,n\in\{-1,0,1\}} \!\!
\frac{\partial G(\xx_i,\xx_j+m\ex+n\ey)}{\partial \nn^{\xx_i}} w_j~,
\label{Anyst}
\ee
where $\delta_{ij}$ is the Kronecker delta.
Here, for entries $i=j$ the standard diagonal limit of the kernel
$\partial G(\xx_i,\xx_i)/\partial \nn^{\xx_i} = -\kappa(\xx_i)/4\pi$
is needed, where
$\kappa(\xx)$ is the signed curvature at $\xx\in\pO$.
The matrix $B\in\RR^{N \times M}$ has entries $B_{ij} = \partial \phi_j(\xx_i)/
\partial \nn^{\xx_i}$.

For the second block row \eqref{row2a}--\eqref{row2d} we use
the above quadrature for the source locations of the operators,
and enforce the four equations on the wall collocation nodes
$\xx_{iL}$, $\xx_{iL}+\ex$, $\xx_{iD}$, $\xx_{iD}+\ey$, $i=1,\dots,m$,
to get
\be
C \btau + Q \xi \; = \; \mbf{g}~,
\label{row2}
\ee
with the macroscopic driving $(p_1,p_2)$ encoded by the right-hand side vector
\be
\mbf{g} \;=\; [p_1\mbf{1}_m;\mbf{0}_m;p_2\mbf{1}_m;\mbf{0}_m] \;\in\; \RR^{4m}
~,
\label{gg}
\ee
where $\mbf{1}_m\in\RR^m$ is the vector of ones.
The matrix $Q$ is precisely as in \eqref{Qsys}.
Rather than list formulae for all four blocks in $C=[C_1;C_2;C_3;C_4]$,
we have
$(C_1)_{ij} = \sum_{n\in\{-1,0,1\}} \bigl(
G(\xx_{iL}+\ex,\xx_j-\ex+n\ey) - G(\xx_{iL},\xx_j+\ex+n\ey)
\bigr) w_j
$, with the others filled analogously.
Stacking the two block rows \eqref{row1} and \eqref{row2} gives the
$(N+4m)$-by-$(N+M)$ extended linear system \eqref{ELS},
which we emphasize is just
a standard discretization of the BVP conditions \eqref{bc2}--\eqref{j2}.

For small problems ($N$ less than a few thousand), the ELS is most simply
solved by standard dense direct methods; for larger problems a
Schur complement must be taken in order to solve iteratively, as presented
in Section~\ref{s:schur}.
First we perform a numerical test of the direct method.

\subsection{Numerical tests using direct solution of the ELS}
\label{s:lapnum}

{\bf Example 1.}
We define a smooth ``worm'' shaped inclusion
which crosses the unit cell walls $L$ and $R$ by $\xx(t) = 
(0.7 \cos t, 0.15 \sin t + 0.3 \sin (1.4 \cos t))$.
The solution of the periodic Neumann Laplace BVP
with external driving $\mbf{p} = (1,0)$ is shown in Fig.~\ref{f:ELS}(a).
Here $u$ is evaluated via \eqref{urep}
both inside $\cb$ (where it is accurate), and, to show the nature of
the representation, out to the proxy circle (where it is certainly
inaccurate). Recall that the proxy sources must represent
the layer potentials due to the infinite lattice of copies which excludes the
$3\times 3$ central block.
Note that two tips of copies in this set slightly penetrate the
proxy circle, violating the condition in Thm.~\ref{t:Qconv} that the
function the proxy sources represent be analytic in the closed
ball of radius $\rp$.
We find in practice that such slight geometric violations do not lead to
problematic growth in $\|\xi\|$, but that larger
violations can induce a large $\|\xi\|$ which limits achievable accuracy.
Panel (c) shows the convergence of errors in $u$ (at the pair of points shown)
to their numerically-converged value, fixing converged values for $M$ and $m$.
Convergence to around 13 digits (for $N=140$, $M=80$) is apparent,
and the solution time is 0.03 seconds.

As an independent verification of the method,
%that the BVP is being solved correctly,
we construct a known solution to a slightly generalized
version of \eqref{lap2}--\eqref{j2},
where \eqref{bc2} is replaced by {\em inhomogeneous} data
$u_n = f$ and a general discrepancy $g$ is allowed.
% as in \eqref{g1}--\eqref{g4}.
We choose the known solution
$$
u_\tbox{ex}(\xx) \; = \; \sum_{m,n\in\{-2,-1,\dots,2\}} \nn_0 \cdot \nabla
G(\xx,\yy_0+m\ex+n\ey)
~,
$$
where the central
dipole has direction $\nn_0$, and location $\yy_0$ chosen inside
$\Omega$ and far from its
boundary (so that the induced data $f$ is smooth).
The grid size of $5\times 5$ (some of which is shown by
$*$ symbols in Fig.~\ref{f:ELS}(a)) is chosen so
that $g$ is sufficiently smooth (which requires at least $3\times 3$), and
so that the periodizing part $\xi$ is nontrivial.
%is so that $g$ is sufficiently smooth to be be accurately
%represented by \eqref{urep}. (In the application, of course,
%the RHS is even smoother, being piecewise constant.)
%The choice of a dipole (as opposed to monopole)
%induces net drops in potential analogous to $(p_1,p_2)$.
From $u_\tbox{ex}$ the right-hand side functions $f$ and $g$ are then evaluated
at nodes,
the ELS solved directly, the numerical solution \eqref{urep} evaluated,
and the difference at two points compared to its known exact value.
The resulting $N$-convergence to 13 digits
is shown in Fig.~\ref{f:ELS}(c).
The convergence rate is slower than before, due to the unavoidable
closeness of the dipole source to $\pO$.
% rate is slower due to bdry data gen by nearby source.

Finally, the grey lines in panel (e) show that, as with $Q$,
there is one singular value of $E$ which is much smaller than
the others, reflecting the unit nullity of the underlying BVP
(Prop.~\ref{p:lapnull}). Also apparent is the fact that, despite the
exponential ill-conditioning,
the solution norm remains bounded once $M$-convergence has occurred.

\bfi % fffffffffffffffffffffffffffffffffffffffffffffffffff
\raisebox{-2.7in}{\ig{width=3in}{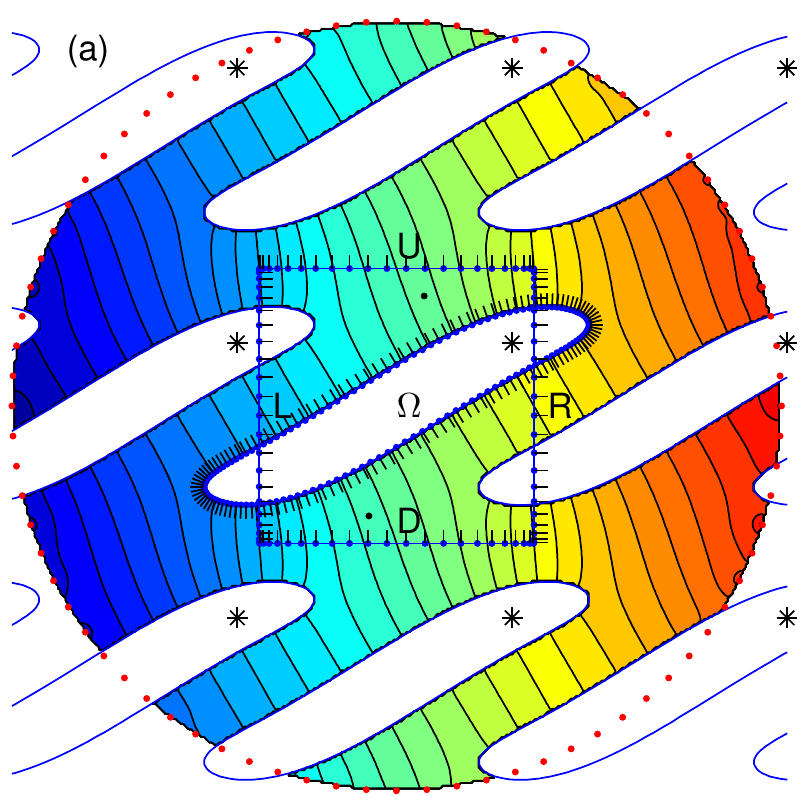}}
\qquad
(b)\;\raisebox{-2.7in}{\ig{width=2.8in}{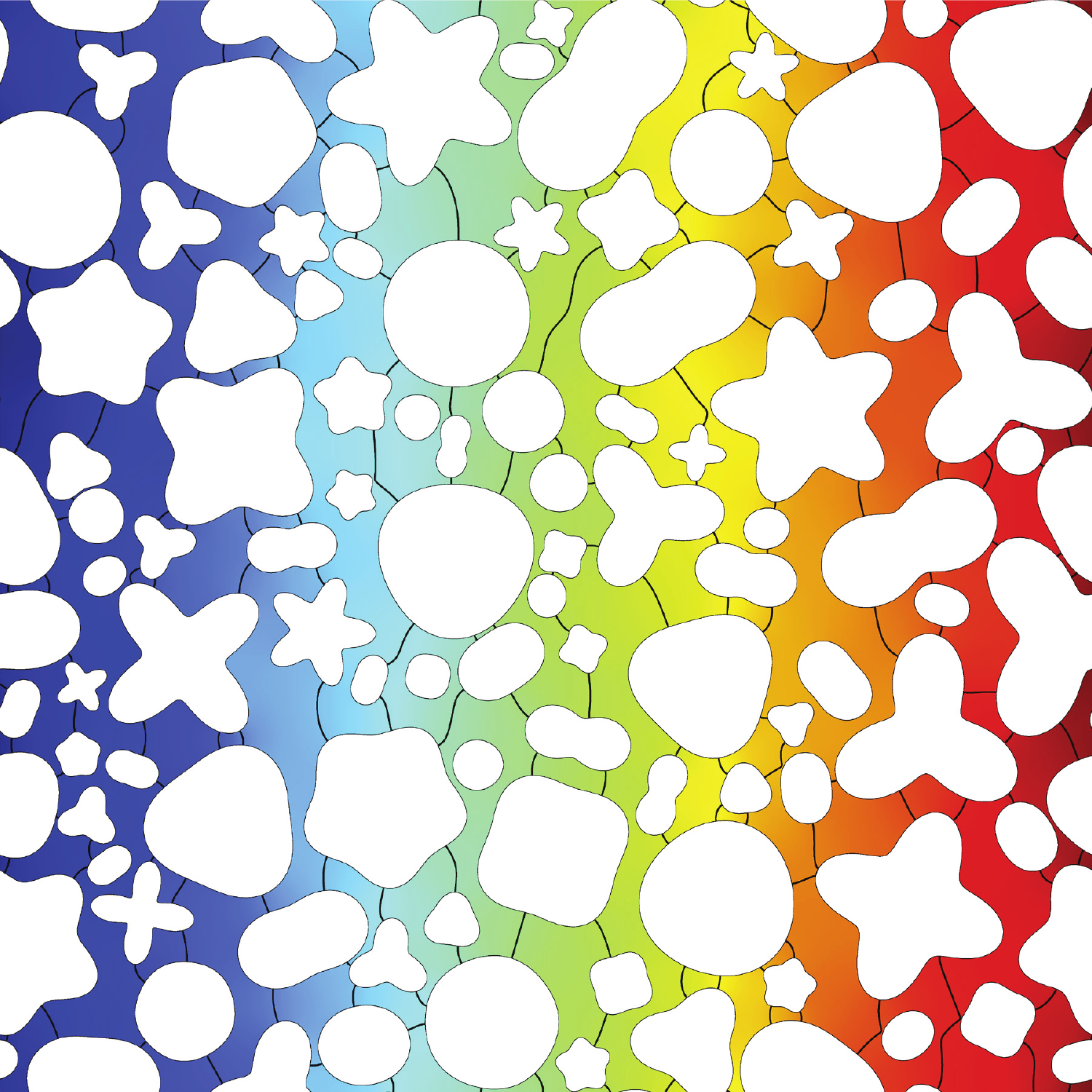}}
%(b) $K=10^3$ sol plot, central unit cell\\
\ig{width=2.1in}{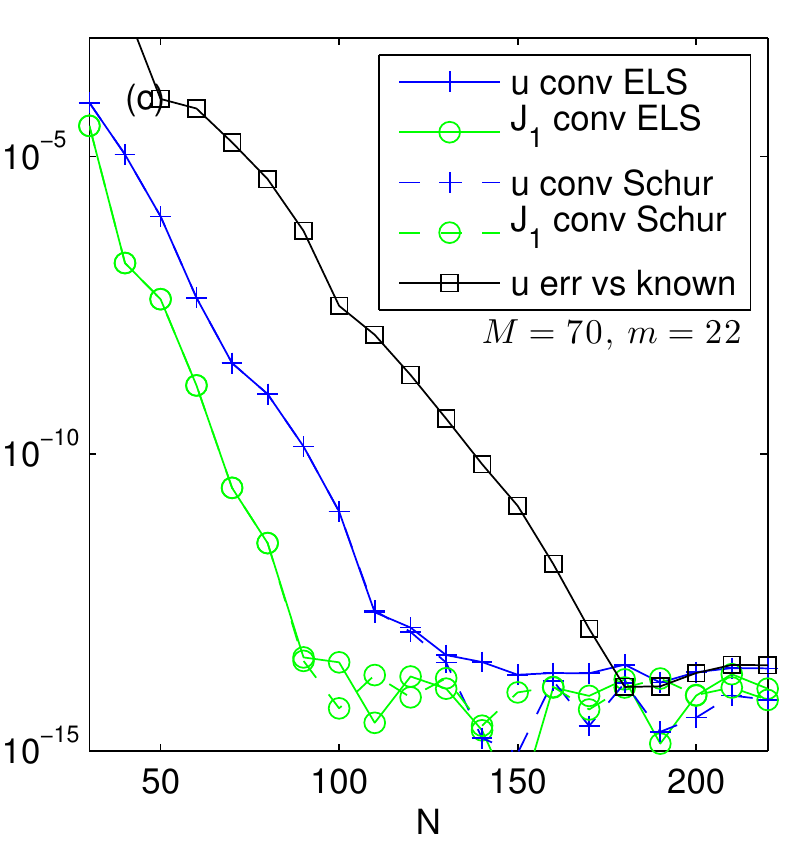}
\ig{width=2.1in}{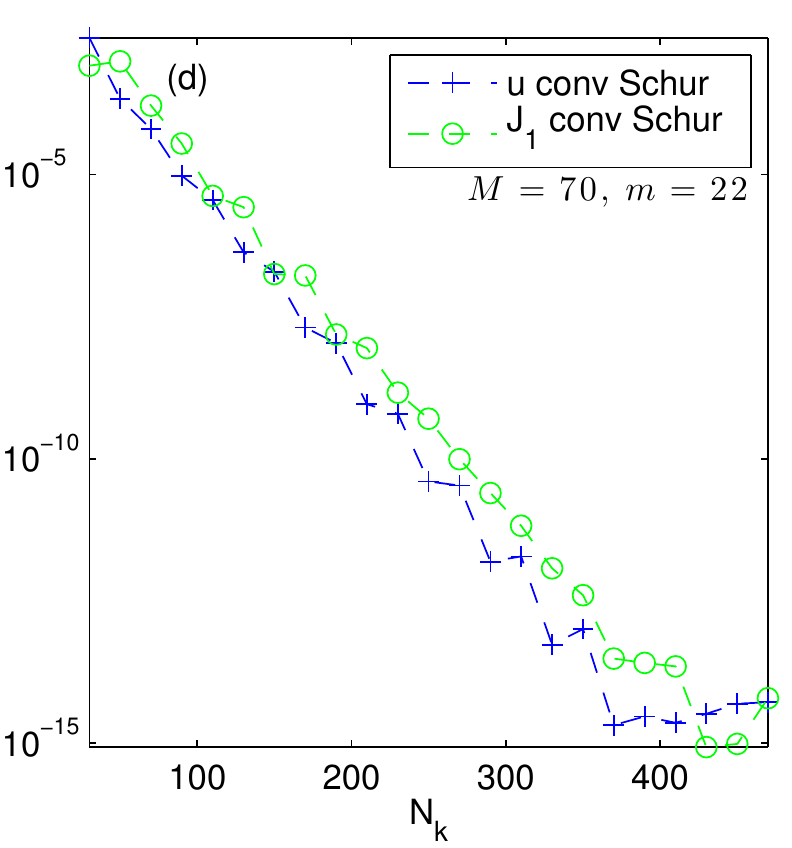}
%(d) $K=10^3$ $N_k$-conv plot
\ig{width=2.1in}{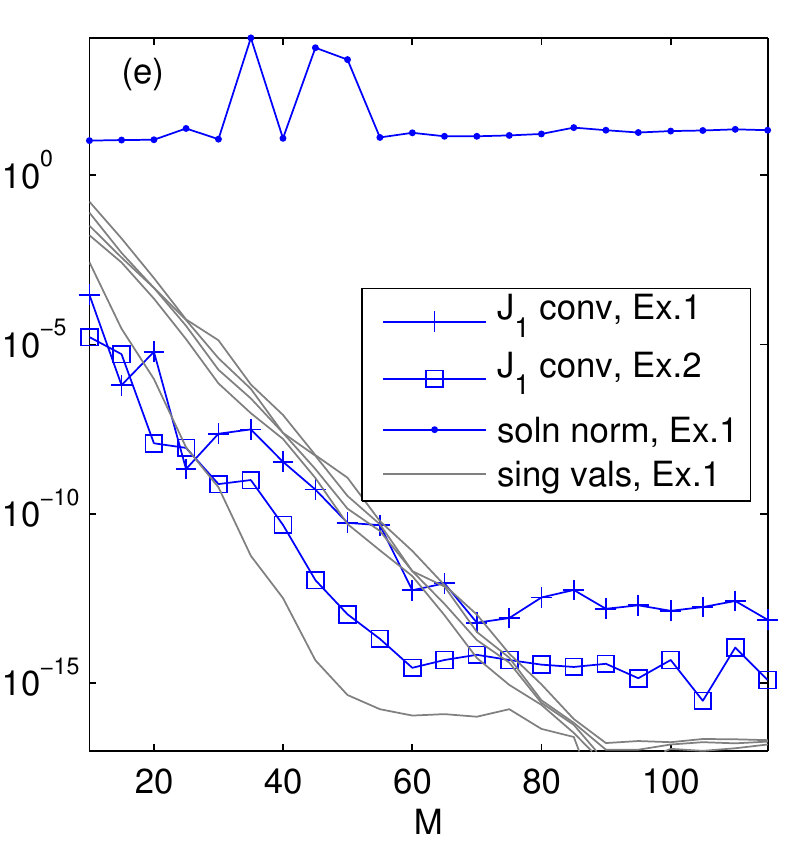}
\ca{Periodic Laplace Neumann tests driven by external driving
$\mbf{p} = (1,0)$.
(a) Solution potential $u$ contours for ``worm'' inclusion (Example 1),
with $N=140$, $M=80$, and $\rp=1.4$.
The $m=22$ nodes per wall and 
nodes on $\pO$ are also shown (with normals),
proxy points (red dots), test points (two black dots),
and $3\times 3$ grid of dipoles generating a known solution (* symbols).
The representation for $u$ is only accurate inside the unit cell.
(b) Solution with $K=100$ inclusions % in the unit cell
(Example 2), with $N_k=400$ unknowns per inclusion,
by iterative solution of \eqref{pschur}.
A single unit cell is shown.
(c)
$N$-convergence of error in difference of $u$ at the two test points,
and of flux $J_1$ computed via \eqref{J1trick},
relative to their values at $N=230$ ($M=80$ is fixed), for Example 1.
Squares show error convergence in $u$ (difference at the test points)
in the case of known $u_\tbox{ex}$ due to the dipole grid.
Solid lines are for direct solution of the ELS, 
dashed lines for the iterative solution of \eqref{pschur}.
(d) Convergence with $N_k$ for Example 2, for
pointwise $u$ (+ symbols) and flux $J_1$ (circles).
% skip any known soln here.
%
(e) $M$-convergence of flux $J_1$ error,
for Example 1 ($+$ signs) and Example 2 (squares), with other
parameters converged.
For $K=1$ the lowest six singular values of $E$ are also shown (grey lines),
and the solution norm $\|[\btau;\xi]\|_2$ (points).
}{f:ELS}
\efi

% ssssssssssssssssssssssssssssssssssssssssssssssssssssssssssssssssssssssss
\subsection{Schur complement system and its iterative solution}
\label{s:schur}

When $N$ is large, solving the ill-conditioned rectangular ELS
\eqref{ELS} is impractical.
We would like to use a Schur complement in the style of \eqref{Aper}
to create an equivalent $N\times N$ system, which
we do in Sec.~\ref{s:equiv}.
Furthermore, in order to use Krylov subspace iterative methods with known
convergence rates,
we would like to remove the nullspace to make this well-conditioned,
which we do in Sec.~\ref{s:wellcond}.
We will do both these tasks
via low-rank perturbation of certain blocks of \eqref{Aper}
before applying the Schur complement.

In what follows, $Q^+$ is the pseudoinverse \cite[Sec.~5.5]{golubvanloan} of $Q$,
i.e.\ the linear map that recovers a small-norm solution (if one exists)
to $Q\xi = \mbf{g}$ via $\xi = Q^+ \mbf{g}$.
The obstacle to using \eqref{Aper} as written is that $Q$ inherits a 
consistency condition from the
empty BVP so that $Q$ has one smooth vector in its left null-space
($\mbf{w}$; see \eqref{wQ}).
However the range of $C$ does not respect this condition,
thus $Q^+ C$ has a huge 2-norm (which we find numerically is at least
$10^{16}$).
We first need to show that the range of a rank-deficient matrix
$Q$ may be enlarged by a rank-$k$ perturbation,
a rectangular version of results about singular square matrices in
\cite{sifuentes2014randomized}.

%Note, however, that our application
%is distinct from the two applications
%presented in \cite{sifuentes2014randomized}---finding nullspaces and
%solving consistent rank-deficient systems.
%In a sense we are using low-rank perturbation to find
%a least-squares solution to an {\em inconsistent} rank-deficient
%system.  % but it's not strict LSQ soln; it's the one orthog to R.
%
%All ranks and nullspaces are numerical.

\begin{lem}  % lllllllllllllllllllllllll
Let $Q\in\RR^{m\times n}$ have a $k$-dimensional nullspace.
Let $R \in\RR^{n\times k}$ have full-rank projection onto $\Nul Q$,
i.e.\ if $N$ has columns forming a basis for $\Nul Q$
then $R\tr N \in \RR^{k\times k}$ is invertible.
Let $V\in\RR^{m\times k}$ be arbitrary.
Then  $\Ran (Q + V R\tr) \supset \Ran Q \oplus \Ran V$,
i.e.\ the range now includes that of $V$.
\label{l:onesmatrix}\end{lem}            % llllllllllllllllllllll
\begin{proof}
We need to check that $(Q+VR\tr)x = Q x_0 + V\alpha_0$
has a solution $x$ for all given pairs $x_0\in\RR^n$, $\alpha_0\in\RR^k$.
Recalling that $R\tr N$ is invertible,
by substitution one may check that
$x = x_0 + N(R\tr N)^{-1} (\alpha_0 - R\tr x_0)$ is an explicit such solution.
\end{proof}

\begin{rmk} With additional conditions $m\ge n$, and that
$V$ has full-rank projection onto a part of $\Nul Q\tr$
(i.e.\ if $W$ has columns forming a basis for a $k$-dimensional
subspace of $\Nul Q\tr$, then $W\tr V \in \RR^{k\times k}$ is invertible),
we get incidentally that $Q+VR\tr$ has trivial nullspace
(generalizing \cite[Sec.~3]{sifuentes2014randomized}).
%although this is not needed
%completing the picture for the removal of the nullity $k$.
The proof is as follows.
Let $x\in\RR^n$ solve the homogeneous equation $(Q + V R\tr)x = 0$.
Then $0 = W\tr Q x = -(W\tr V) R\tr x$, but $W\tr V$ is invertible, so
$R\tr x = 0$. Thus the homogeneous equation becomes $Qx =0$, which means
there is an $\alpha \in \RR^k$ such that $x = N\alpha$.
Thus $R\tr N \alpha = 0$, but $R\tr N$ is invertible, so that $\alpha=0$,
so $x=0$. %Thus $\Nul \tilde Q = \{0\}$.
\end{rmk}

We also need the fact that a block-column operation allows $Q$ to be perturbed
as above while changing the ELS solution space in a known way.
The proof is simple to check.
\begin{pro}
Let $A$, $B$, $C$ and $Q$ be matrices, and let $P$ be a matrix
with as many rows as $A$ has columns and as many columns as $B$ has rows.
Then the pair $(\ttau, \xi)$
solves the block system
\be
\mt{A}{B + AP}{C}{Q+CP}\vt{\ttau}{\xi} \; = \; \vt{\mbf{0}}{\mbf{g}}~
\label{gary}
\ee
if and only if the pair $(\bm\tau, \xi)$, with $\bm\tau = \ttau + P \xi$,
solves the original ELS \eqref{ELS}.
%where $\mbf{0}$ and $\mbf{g}$ are column vectors of the same heights as $A$ and $C$.
\label{p:gary}\end{pro}
%\begin{proof} Simply insert the identity written as $[I P; 0 I][I -P; 0 I]$ after the system matrix in \eqref{ELS}. \end{proof}

\subsubsection{An equivalent square system preserving the nullspace}  % ssssssss
\label{s:equiv}

Armed with the above, a method to ``fold'' the
Neumann Laplace ELS into an equivalent
%for the periodic Neumann Laplace BVP
square system is as follows:
\ben
\item Set the proxy coefficient vector $R = (1/M)\mbf{1}_M$,
and the discrete density vector $H = \mbf{1}_N$.
Create low-rank perturbed matrix blocks
$\tilde B := B + A H R\tr$ and $\tilde Q := Q + C H R\tr$.
\item Solve for the vector $\ttau$ in the $N\times N$ Schur
complement linear system
\be
(A - \tilde B \tilde Q^+ C)\ttau \; = \; -\tilde B \tilde Q^+ \mbf{g}~.
\label{Gschur}
\ee
%which may be done iteratively.
More precisely, since (due to the numerical ill-conditioning in $\tilde Q$)
multiplying by $\tilde Q^+$ would lose accuracy
due to rounding error, instead solve the small systems
$\tilde Q X = C$  for $X$, and $\tilde Q \mbf{y} = \mbf{g}$ for $\mbf{y}$,
heeding Remark~\ref{r:linsolve}.
From them, build $A - \tilde B X$ and $-\tilde B \mbf{y}$,
which are respectively the
% large
system matrix and right-hand side for \eqref{Gschur}.
This large square
system may then be solved iteratively (see section~\ref{s:multi}).
%Within a Krylov subspace iterative solver,
\item Recover the proxy coefficients via $\xi = \mbf{y} - X \ttau$.
\item Recover the density via $\bm\tau = \ttau + HR\tr \xi$.
\een
%The resulting pair $(\bm\tau,\xi)$ solves original ELS \eqref{ELS}.
%Why does this procedure work?

Note that in Step 1 the prefactor $1/M$ leads to correct quadrature scaling,
so that $HR\tr$ has similar 2-norm to the other matrices
(also recommended in \cite{sifuentes2014randomized}).
% *** say more re numerics tweak? later?

\begin{thm}   % tttttttttttttttttttttttttttttttttt
% include that g consistent? no, leave a special driving case.
% Since Aper singular here, won't have solution if g not cons.
Let $\mbf{g}$ encode the driving as in \eqref{gg}.
Then for all sufficiently large $N$, $M$, and $m$,
any pair $(\btau,\xi)$ produced by the
above procedure performed in exact arithmetic solves the
original ELS \eqref{ELS}.
\label{t:gary}\end{thm}
% save sth re any consistent g possible? (but we haven't defined
% consistency for g in the original BVP)
\begin{proof}
By rotational invariance, a constant single-layer density on a circle generates
a constant potential inside, and 
inserting $R$ into \eqref{mfs} gives the
periodic trapezoid quadrature approximation to such a potential, thus
generates discrepancy near zero (in fact exponentially small in $M$).
Thus for all sufficiently large $M$, $R$ is not orthogonal to $\Nul Q$.
Applying Lemma~\ref{l:onesmatrix} with $k=1$ gives that the
range of $\tilde Q$ includes the discrepancy vector $V = CH$
produced by the constant density $H$.
Thus, to show that the range of $\tilde Q$ includes all smooth vectors,
i.e.\ that it does not have the consistency condition \eqref{wQ}, %for $Q$,
one needs to check that
$\mbf{w}\tr V= \mbf{w}\tr C H \neq 0$
which is done in Lemma~\ref{l:onesdens} below.
%has no consistency condition
Thus the systems $\tilde Q X = C$ and $\tilde Q \mbf{y} = \mbf{g}$ are
consistent for any $C$ and $\mbf{g}$,
so that the Schur complement is well defined.
Finally, Prop.~\ref{p:gary}, using the rank-1 matrix $P = H R\tr$,
insures that step 4 recovers a solution to \eqref{ELS}.
\end{proof}

For the missing technical lemma,
we first need a form of Gauss' law, stating that
for any density $\tau$ on a curve $\pO$
the single-layer potential $v = {\cal S}_\pO\tau$ generates flux
equal to the total charge, i.e.\
\be
\int_{\partial \cal K} v_n = \int_\pO \tau\, ds
\label{Gauss}
\ee
where $\partial \cal K$ is the boundary of some open domain $\cal K$
containing $\pO$.
This may be proved by combining the jump relation
$v_n^+ - v_n^- = -\tau$ (from \eqref{JR}) with the fact that,
since $v$ is a Laplace solution in $\RR^2\backslash \pO$,
$\int v_n = 0$ taken over the boundaries of $\Omega$ and of
$\cb \backslash \overline{\Omega}$.

%The missing piece is the following, which is a special case of the
%statement that the consistency condition is the total charge on $\pO$.
\begin{lem}
Let the matrix $C$ be defined as
in Section~\ref{s:lapbie}, and $\mbf{w}$ be as in \eqref{w}.
Then, for all $N$ and $m$ sufficiently large, $\mbf{w}\tr C \mbf{1}_N \neq 0$.
\label{l:onesdens}
\end{lem}
\begin{proof}
Let $d = [d_1;d_2;d_3;d_4]$ be the discrepancy of the potential
$v = {\cal S}^\tbox{near}_\pO \tau$ for density $\tau\equiv 1$.
As used in the proof of Prop.~\ref{p:lapemptynull},
the flux (left hand side of \eqref{Gauss}) out of the unit cell $\cb$
equals $\int_L d_2 ds + \int_D d_4 ds$,
which by \eqref{Gauss} (and noting that only one of the
nine source terms in ${\cal S}^\tbox{near}_\pO$ lies within $\cb$)
equals the perimeter $\int_\pO 1 ds =|\pO|> 0$.
What we seek is the discretization of this statement about the PDE.
If $\mbf{d}$ is the discrepancy of $v$ sampled at the wall
nodes, then %by construction of the $C$ matrix
its %Nystr\"om
quadrature approximation is $\mbf{d} \approx C \mbf{1}_N$,
and so (as discussed above \eqref{w}) the quadrature
approximation to $\int_L d_2 ds + \int_D d_4 ds$
is $\mbf{w}\tr \mbf{d} \approx \mbf{w}\tr  C \mbf{1}_N$.
Thus, as $N$ and $m$ tend to infinity,
the latter converges to $|\pO|$.
\end{proof}

Theorem~\ref{t:gary}
justifies rigorously one procedure to create a square system
equivalent to the ELS \eqref{ELS}.
However, by the equivalence in Prop.~\ref{p:gary}, the system matrix
$A - \tilde B X$ at the heart of the procedure is singular,
as it inherits the unit nullity of the ELS, which itself
derives from the unit nullity of the Laplace Neumann BVP.
Since the convergence of iterative solvers for singular
matrices is a subtle matter \cite{singgmres}, 
this motivates
%we will only test numerically 
the following improved variant which removes the nullspace.

\subsubsection{A well-conditioned square system} % sssssssssssssssssssssssssssss
\label{s:wellcond}
The following simpler variant creates a non-singular square system
from the ELS; its proof is
more subtle. It is what we recommend for the iterative solution of the periodic Neumann Laplace problem, and test numerically:
\ben
\item Set the constant proxy coefficient vector $R = (1/M)\mbf{1}_M$, the constant discrete density vector $H = \mbf{1}_N$,
and $\tilde Q := Q + C H R\tr$.
\item Solve for the vector $\bm\tau$ in the $N\times N$ Schur
complement linear system
\be
\tilde A_\tbox{per} \btau := (A - B \tilde Q^+ C)\bm\tau \; = \;
-B \tilde Q^+ \mbf{g}~,
\label{pschur}
\ee
%which may be done iteratively.
where, as before, one solves the small systems
$\tilde Q X = C$ and $\tilde Q \mbf{y} = \mbf{g}$,
to build the large system matrix $\tilde A_\tbox{per} = A - B X$
and right-hand side
$-B \mbf{y}$ for \eqref{pschur},
which may then be solved iteratively (see section~\ref{s:multi}).
%Within a Krylov subspace iterative solver,
\item Recover the proxy coefficients via $\xi = \mbf{y} - X \bm\tau$.
\een
\begin{thm} % ttttttttttttttttttttttttttttttttttttttttttttttttttttttttttttttt
Let $\mbf{g}$ encode the driving as in \eqref{gg}.
Then for all $N$, $M$ and $m$ sufficiently large, the pair
$(\btau,\xi)$ produced by the above procedure is unique,
and solves the ELS \eqref{ELS}
with residual of order the quadrature error on boundaries.
\label{t:alex}\end{thm}            % tttttttttttttttttt
% Note: it happens that consistency for a general g in original BVP is same
% conds as for empty BVP, ie in the sense of Sec.~\ref{s:lapempty}.
% But, we stick to special case of g given by p1,p2 for simplicity.
\begin{proof}
First note that \eqref{pschur} is the Schur complement of
the perturbed ELS
\be
\mt{A}{B}{C}{Q+CHR\tr}\vt{\btau}{\xi} \; = \; \vt{0}{\mbf{g}}
~.
\label{pELS}
\ee
In particular one may check that if $\btau$ solves \eqref{pschur}
then $(\btau,\xi)$, with $\xi$
% = \tilde Q^+ \mbf{g} - \tilde Q^+ C \btau$
as in Step 3, solves \eqref{pELS}.
%and thus equivalent to
Given such a solution $(\btau,\xi)$, define the potential generated
by the usual representation
$$
v (\xx) := 
\sum_{m,n\in\{-1,0,1\}} \sum_{i=1}^N w_i G(\xx,\xx_i+m\ex+n\ey) \tau_i
+ \sum_{j=1}^M \xi_j \phi_j(\xx)~,
$$
i.e.\ the quadrature approximation to \eqref{urep}.
Note that $C \btau + Q \xi$ then approximates the discrepancy
of $v$ at the nodes on the unit cell walls.
The first row of \eqref{pELS} states that $v_n = 0$ on $\pO$,
and since $v$ is harmonic in $\cb \backslash \overline{\Omega}$,
the net flux of $v$ through $\partial\cb$ is zero.
This means that the discrepancy of $v$ obeys the same consistency
condition as in Prop.~\ref{p:lapemptynull}, which when discretized
on the walls gives $\mbf{w}\tr (C \btau + Q \xi) = 0$,
where $\mbf{w}$ is defined by \eqref{w}.
Subtracting this from the second row of \eqref{pELS}
left-multiplied by $\mbf{w}\tr$ leaves only the expression
$(\mbf{w}\tr C H) R\tr\xi = \mbf{w}\tr \mbf{g}$.
By Lemma~\ref{l:onesdens}, $\mbf{w}\tr C H \neq 0$,
and by computation $\mbf{w}\tr \mbf{g} = 0$.
Thus $R\tr\xi = 0$, so the pair also solves the original ELS \eqref{ELS}.
Thus Lemma~\ref{l:laptau} (strictly, its quadrature approximation) holds,
so $\btau$ is unique.
\end{proof}

In contrast to the previous section,
the system \eqref{pschur} to be solved
is {\em well conditioned} if the non-periodic BIE matrix $A$
is; %since in essence
the unit nullity of \eqref{ELS} has been removed by
imposing one extra condition $R\tr \xi = 0$.

We now test the above procedure for
the Laplace Neumann periodic BVP of Example 1 (Fig.~\ref{f:ELS}(a)).
%with periodic ``worm'' inclusions
%recalling that %$Q\in\RR^{4m\times M}$, and that
%$\mbf{g} = [p_1\mbf{1}_m;\mbf{0}_m;p_2\mbf{1}_m;\mbf{0}_m]$
%encodes the pressure drop $(p_1,p_2)$.
We solve \eqref{pschur} iteratively via GMRES
with a relative stopping residual tolerance of $10^{-14}$.
In Step 2 we use {\tt linsolve} as in Remark~\ref{r:linsolve},
and verify that the resulting norm $\|X \|_2 \approx 9$ is not large.
Fig.~\ref{f:ELS}(c) includes (as dashed lines) the
resulting self-convergence of $u$, and of the flux $J_1$
(computed as in Sec.~\ref{s:kappa}),
with other parameters converged as before.
%with the same $m=22$ and $M=80$ as before.
The converged values agree to around $10^{-13}$.
Above $10^{-13}$, the errors are identical to those of the full ELS.
%whereas below there is less variation with $N$. No longer!
The condition number of $\tilde A_\tbox{per}$ is 8.3, and the number of
GMRES iterations required was 12, both independent of $N$ and $M$.

\bfi %% fffffffffffffffffffffffffffffffffffffffffffffffffffffffff
\mbox{%m
\hspace{-1ex}
\ig{height=1.5in}{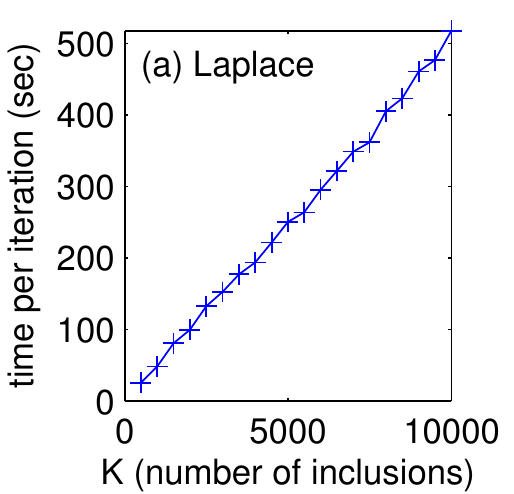}
\ig{height=1.5in}{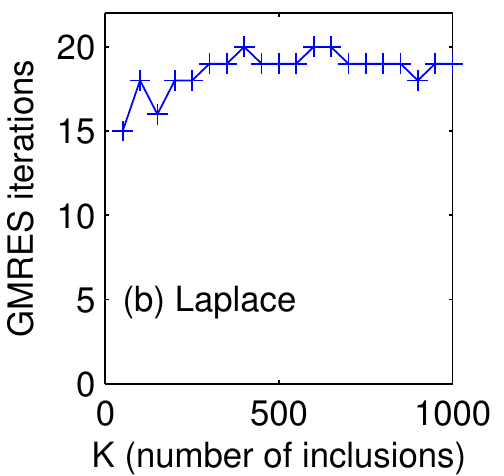}
\ig{height=1.5in}{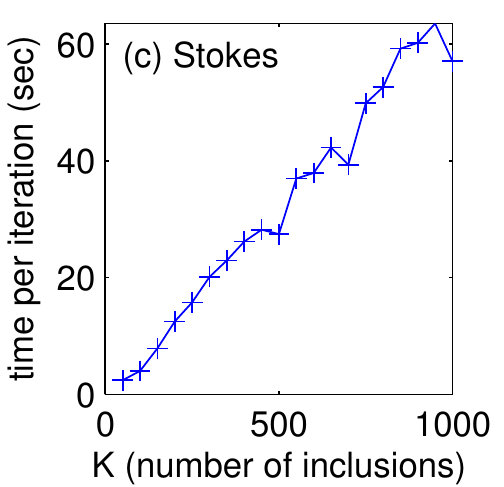}
\ig{height=1.5in}{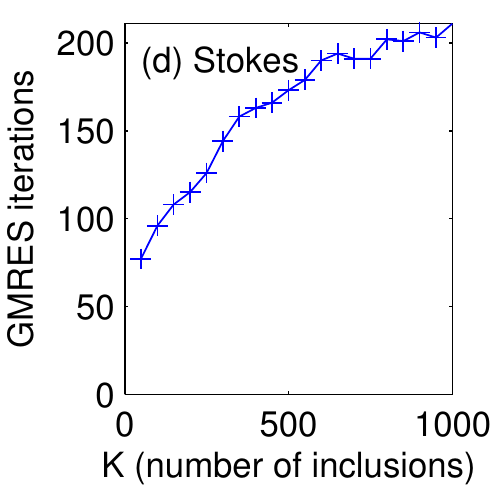}
}%m
\ca{Scaling of computation time per iteration, and iteration counts,
with $K$, the number of inclusions.
%for Laplace and Stokes.
In all four plots, we set $N_k=256$.
(a) CPU time per iteration for the Laplace problem with random
star-shaped obstacles as in Figs.~\ref{f:lapcandy} and \eqref{f:ELS}(b).
(b) Number of GMRES iterations to reach a relative residual of
$\epsilon=10^{-14}$, for circles with $f_\tbox{clup}=10$.
(c) and (d) are the same as (a) and (b), but for the Stokes problem.
}{f:GMRES}
\efi

% KKKKKKKKKKKKKKKKKKKKKKKKKKKKKKKKKKKKKKKKKKKKKKKKKKKKKKKKKKKKKKKKKKKKKKKKKKK
\subsection{Multi-inclusion examples and close-to-touching geometries}
\label{s:multi}

Generalizing the above to $K>1$ disjoint inclusions $\{\Omega_k\}_{k=1}^K$
in the unit cell is largely a matter of bookkeeping.
The representation \eqref{urep} becomes
a $3\times 3$ image sum over single-layer potentials
on each inclusion boundary,
\be
u \; = \; \sum_{k=1}^K{\cal S}^\tbox{near}_{\pO_k}\tau_k + \sum_{j=1}^M \xi_j \phi_j
~, \qquad \mbox{ with } \; \tau := \{\tau_k\}~. 
\label{urepK}
\ee
In particular, the proxy representation is unchanged.
%involves no more unknowns.
Upon discretization using $N_k$ nodes on the $k$th inclusion boundary,
with a total number of unknowns $N:=\sum_{k=1}^K N_k$,
the $A$ matrix now has a $K$-by-$K$ block structure
where the $-\half$ identity only appears in the diagonal blocks.
For large $N$,
to solve the resulting linear system \eqref{pschur} iteratively via GMRES,
one needs to apply $\tilde A_\tbox{per} = A-BX$ to any given
vector $\btau$.
We can perform this matrix-vector multiply in $\bigO(N)$ time, as follows.
We apply the off-diagonal of $A$ using the FMM with source
charges $\tau_j w_j$, and the 
diagonal of $A$ as discussed below \eqref{Anyst}.
Then, having pre-stored $B$ and $X$, which needs $\bigO(MN)$ memory,
we compute the correction $-B(X\btau)$ using two 
standard BLAS2 matrix-vector multiplies.

When curves come close (in practice, closer than $5h$, where $h$ is the local
node spacing \cite[Rmk.~6]{ce}), high accuracy demands
replacing the native Nystr\"om quadrature \eqref{Anyst}
with special quadrature formulae.
Note that this does not add extra unknowns.
For this we use recently developed close-evaluation quadratures for the
periodic trapezoid rule with the Laplace single-layer potential
\cite{lsc2d}.
For each of the source curves for which a given target is within $5h$,
we subtract off the native contribution from the above FMM evaluation
and add in the special close evaluation for that curve.

\begin{rmk}[Geometry generation] % rrrrrrrrrrrrrrrrrrrrrrrrrrrrrrrrrrrrrrrrrrrr
In all our remaining numerical examples except Examples 3 and 5, we
create random geometries
with a large number $K$ of inclusions as follows.
We generate polar curves of the form $r(\theta) = s(1 + a\cos(w\theta + \phi))$,
with $\phi$ random, $a$ uniform random in $[0,0.5]$, $w$ randomly
chosen from the set $\{2,3,4,5\}$, and $s$ varying over a size ratio of $4$.
Starting with the largest $s$, we add in such curves (translated to uniformly
random locations in the unit cell), then discard any that intersect.
This is repeated in a sequence of decreasing $s$ values
until a total of $K$ inclusions are generated.
Finally, the $s$ (size) of all inclusions were multiplied by 0.97.
This has the effect of making a random geometry with a
minimum relative closeness
of around 3\% of the radius (this is only approximate since it depends on
local slope $r'(\theta)$).
Helsing--Ojala \cite{helsing_close} defines for circles a closeness parameter
$f_\tbox{clup}$ as the upper bound on the 
the circumference of the larger circle divided by the minimum distance
between the curves. If in their definition
one replaces circumference by $2\pi$ times the largest radius
of a non-circular curve, then in our geometry $f_\tbox{clup} \approx 200$.
\label{r:geom}
\end{rmk}        % rrrrrrrrrrrrrrrrrrrrrrrrrrrrrrrrrrrrrrrrrrrrrrrrrrrrrrrrrrr

{\bf Example 2.}
With $K=100$ inclusions generated as just described,
and an external driving $\mbf{p}=(1,0)$,
we use the well-conditioned iterative method from sec.~\ref{s:wellcond},
and the Laplace FMM of Gimbutas--Greengard \cite{HFMM2D}.
Fig.~\ref{f:ELS}(b) shows the solution potential, and
Fig.~\ref{f:ELS}(d) the self-convergence of the potential
at a point and of the flux, with respect to $N_k$.
Both achieve at least 13 digits of accuracy.
At a fully-converged $N_k=400$, the solution time was 445 seconds.
Fig.~\ref{f:ELS}(e) shows the convergence with respect to
$M$, the number of proxy points:
this confirms that this convergence rate is at least as good as
it is for $K=1$.
In other words, as least for a square unit cell, the $M$ required for close to machine accuracy is around 100,
and, as expected, is {\em independent of the complexity of the geometry}.
%Regardless of the size of $N$,
%for high accuracy in the case of a square unit cell and geometry that
%does not penetrate much beyond the unit cell walls,
%$M$ need never be larger than 100.

For this example,
we verify linear complexity of the scheme in Fig.~\ref{f:GMRES}(a),
for up to $K=10^4$ inclusions ($N=2.56 \times 10^6$ total unknowns).
In Fig.~\ref{f:lapcandy}, we plot the solution for
$K=10^4$ inclusions,
and $N_k=512$ (or $N=5.12 \times 10^6$ total unknowns).
The solution
requires 77 GMRES iterations and around 28 hours of computation time.
The flux error is estimated at $10^{-4}$, by comparing
to the solution at a larger $N_k$ value.
The cause of this lower accuracy compared to that achieved at smaller $K$ is
an area for future research.
%%%%%%% ******* REPLACE THESE WHEN KNOW BETTER

{\bf Example 3.}
A natural question is how the complexity of the geometry
%to be periodized
affects the number of GMRES iterations.
To address this, we generate simpler random geometries using $K$
circles, again with random sizes of ratio up to 4,
but with $f_\tbox{clup} = 10$,
which means that curves are not very close to each other.
In Fig.~\ref{f:GMRES}(b) shows that the number of iterations
grows very weakly, if at all, with $K$.
The interesting question of the impact of $f_\tbox{clup}$ on iteration
count we postpone to future work, but note that this has
been studied in the non-periodic case \cite{helsing_close} and would
expect similar results.

\bfi %fffffffffffffffffffffffffffffffffffffffffffff
\ig{width=6.3in}{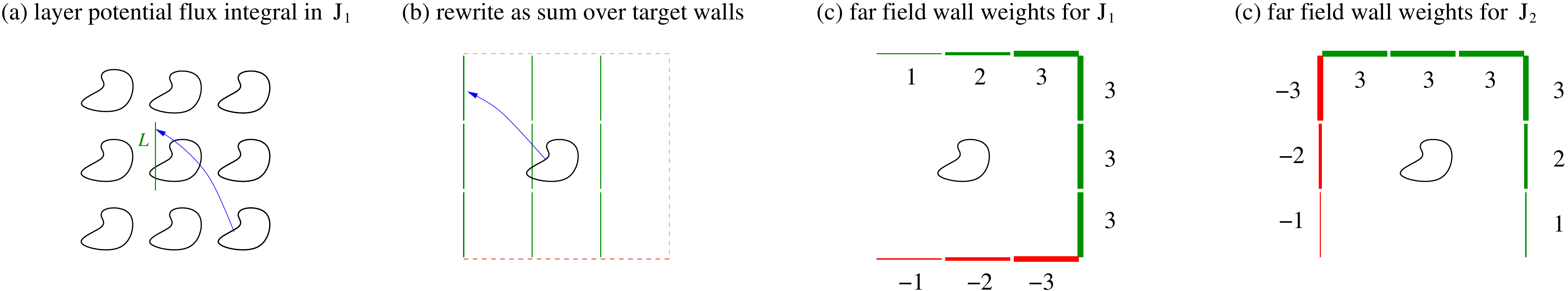}
\ca{Efficient evaluation of fluxes $(J_1,J_2)$ using far field interations
alone.
(a) Nine terms in $L$ wall integral in $J_1$ from the $3\times 3$
layer potential image sum in \eqref{urep}.
(b) Re-interpretation as a sum over targets (nine copies of $L$)
for the potential $v$.
The red dotted line shows the closure of the boundary
where flux conservation is applied.
(c) Resulting weights of the flux integrals of $v$ on nine wall segments;
note all are distant from $\pO$.
(d) The wall weights for $J_2$.
}{f:fluxfar}
\efi   % fffffffffffffffffffffffffffffffffffffffffffffffff

% ccccccccccccccccccccccccccccccccccccccccccccccccccccccccccccccccccccccccccc
\subsection{Computing the effective conductivity tensor}
\label{s:kappa}

An important task is to
compute the effective conductivity $\kappa \in \RR^{2\times 2}$,
which expresses how the mean flux depends on the driving.
%encoding the complete macroscopic behavior of the medium.
Let $\mbf{J}:=(J_1,J_2)$ be the mean flux, with components
%may be evaluated by integrating along
\be
J_1 := \int_L u_n ds~,\qquad \mbox{ and} \;\;
J_2 := \int_D u_n ds
~,
\label{flux}
\ee
and, recalling the pressure vector $\mbf{p}=(p_1,p_2)$,
the conductivity tensor is defined by Darcy's law
\be
\mbf{J} = \kappa \mbf{p}~.
\label{kappa}
\ee
As is well known \cite{miltonbook,Moura94},
to extract the four elements of $\kappa$, it is sufficient
to solve two BVPs (``cell problems''),
one with $p_1=1,p_2=0$ (from which one may read off
$\kappa_{11} = J_1$ and $\kappa_{21} = J_2$),
and the other with $p_1=0,p_2=1$ (and read off $\kappa_{12} = J_1$ and
$\kappa_{22} = J_2$).
Note that $\kappa$ is symmetric \cite[Cor.~6.10]{cioranescu},
hence
$|\kappa_{12}-\kappa_{21}|$ provides an independent gauge of numerical accuracy.

For large-scale problems,
approximating the integrals \eqref{flux} directly by quadrature on the walls
$L$ and $D$ is inconvenient, because, when inclusions intersect the
walls, this forces the integral to be broken into intervals
and forces close-evaluation quadratures to be used.
%To avoid this,
One could deform the integration paths to avoid inclusions,
but finding such a smooth path is complicated, and needs
many quadrature nodes, due to having to pass near inclusions.
Instead, we propose the following method which pushes all interactions to
the far field, and thus requires only a fixed
$m\approx 20$ target nodes per wall and no special quadratures.
%which exploits the representation \eqref{urep}.
\begin{pro}     % pppppppppppppppppppppppppppppppppppppppppp
Let $u$ be represented by \eqref{urep} and solve the BVP
\eqref{lap2}--\eqref{j2} in $\cb$.
Define $v = {\cal S}_\pO\tau$. Then the horizontal flux in \eqref{flux}
can be written
\be
J_1 = 
%\frac{1}{H}\biggl[
\sum_{m\in\{-1,0,1\}} \!\! (m+2)\biggl(
\int_{U+m\ex+\ey} \hspace{-5ex} v_n ds - \int_{D+m\ex-\ey} \hspace{-5ex} v_n ds
\biggr)
+ 3 \!\! \sum_{n\in\{-1,0,1\}} \!\! \int_{R+\ex+n\ey} \hspace{-5ex} v_n ds
\;+\;
\int_L \sum_{j=1}^M \frac{\partial \phi_j}{\partial n} \xi_j ds
%\biggr]
~.
%\sum_{m\in\{-1,0,1\}} \!\!
%(m+2) (D\tr_{U+m\ex+\ey,\pO} - D\tr_{D+m\ex-\ey,\pO})
%D\tr_{U-\ex+\ey,\pO} + 2 D\tr_{U+\ey,\pO} + 3 D\tr_{U-\ex+\ey,\pO}
%- D\tr_{D-\ex-\ey,\pO} - 2 D\tr_{D+\ey,\pO} - 3 D\tr_{D-\ex-\ey,\pO}
%\nonumber \\
%+ \; 3 \!\! \sum_{n\in\{-1,0,1\}} \!\! D\tr_{R+\ex+n\ey,\pO}
%\biggr] \tau
\label{J1trick}
\ee
\end{pro}
The flux integrals involving $v$ are on the {\em distant} walls of the
$3\times 3$ ``super-cell'', with locations and weights shown in
Fig.~\ref{f:fluxfar}(c).
The final term involves a smooth integrand on the original wall $L$.
A similar far-field formula for $J_2$ is achieved by
reflection through the line $x_1=x_2$, with weights shown in
Fig.~\ref{f:fluxfar}(d).
\begin{proof}
Substituting \eqref{urep} into $J_1$ in \eqref{flux}
involves a $3\times 3$ sum over density sources (Fig.~\ref{f:fluxfar}(a))
which by translational invariance
we reinterpret as a sum over displaced target copies as in panel (b).
These nine copies of $L$ form three continuous vertical walls.
Since $v$ is harmonic in $\RR^2\backslash \overline{\Omega}$,
then $\int_\Gamma v_n = 0$ for any closed curve $\Gamma$
that does not enclose nor touch $\pO$.
However, since $u_n=0$ on $\pO$ and $u-v$ is harmonic in a neighborhood of
$\Omega$, this also holds if $\Gamma$ encloses $\pO$.
Thus the flux through each length-3 vertical wall is equal to the
flux through its closure
to the right along the dotted contour shown in panel (b).
Summing these three contour closures, with appropriate normal senses,
gives the weights in panel (c), i.e.\ \eqref{J1trick}.
One may check that the result is unaffected by intersections of $\pO$ with
the original unit cell walls.
\end{proof}

We have tested that this formula matches the naive quadrature of \eqref{flux}
when $\pO$ is far from $L$.
For Example 1,
Fig.~\ref{f:ELS}(c) and (d) include convergence plots for the
flux $J_1$ as computed by \eqref{J1trick},
showing that it converges at least as fast as do pointwise potential
values.
For the parameters of panel (a) we find that
$|\kappa_{12}-\kappa_{21}| = 6 \times 10^{-14}$, 
indicating high accuracy of the computed tensor.

% kkkkkkkkkkkkkkkkkkkkkkkkkkkkkkkkkkkkkkkkkkkkkkkkkkkkkkkkkkkkkkkkkkkkkkkkkkkk
% kkkkkkkkkkkkkkkkkkkkkkkkkkkkkkkkkkkkkkkkkkkkkkkkkkkkkkkkkkkkkkkkkkkkkkkkkkkk
\section{The no-slip Stokes flow case}
\label{s:sto}

We now move to our second BVP, that of viscous flow through
a periodic lattice of inclusions with no-slip boundary conditions.
%As before we assume a single inclusion $\Omega$ per unit cell,
%forming the lattice $\Omega_\Lambda$.
We follow the normalization and some of the notation of
\cite[Sec. 2.2, 2.3]{HW}.
%which seems to be based on \cite{Ladyzhenskaya} except that it includes the 3D cases.
Let the constant $\mu>0$ be the fluid viscosity. The periodic BVP is to
solve for a velocity field $\uu$ and pressure $p$ function
satisfying
\bea
-\mu\Delta\uu + \nabla p &=&0
\qquad \mbox{ in } \RR^2 \backslash \overline{\Omega_\Lambda}
\label{sto}
\\
\nabla\cdot\uu &=& 0
\qquad \mbox{ in } \RR^2 \backslash \overline{\Omega_\Lambda}
\label{incomp}
\\
\uu & = & \mbf{0} \qquad \mbox{ on } \partial \Omega_\Lambda
\label{noslip}
\\
\uu(\xx+\ex) &=& \uu(\xx+\ey) \; = \; \uu(\xx) \qquad \mbox{for all } \xx \in \RR^2 \backslash \overline{\Omega_\Lambda}
\label{uper}
\\
p(\xx+\ex) - p(\xx) &=& p_1 \qquad \mbox{for all } \xx \in \RR^2 \backslash \overline{\Omega_\Lambda}
\label{pr1}
\\
p(\xx+\ey) - p(\xx) &=& p_2 \qquad \mbox{for all } \xx \in \RR^2 \backslash \overline{\Omega_\Lambda}
~.
\label{pr2}
\eea
The first two are Stokes' equations, expressing local force balance and
incompressibility, respectively. The third is the no-slip condition,
and the remainder express that the flow is %spatially
periodic and
the pressure periodic up to the given macroscopic pressure driving $(p_1,p_2)$.

We recall some basic definitions. % \cite[Sec. 2.2, 2.3]{HW}.
Given a pair $(\uu,p)$ the Cauchy stress tensor field has entries
\be
\sigma_{ij}(\uu,p):= -\delta_{ij} p + \mu(\partial_i u_j + \partial_j u_i)
~, \quad i,j = 1,2~.
\label{sigma}
\ee
The hydrodynamic traction $\Tv$ (force vector per unit length that a boundary
surface with outwards unit normal $\nn$ applies to the fluid),
also known as the Neumann data, has components
\be
T_i(\uu,p) := \sigma_{ij}(\uu,p)n_j =
-pn_i + \mu(\partial_i u_j + \partial_j u_i) n_j
~,
\label{T}
\ee
where here and below
the Einstein convention of summation over repeated indices is used.
We first show that 
%As with the Laplace Neumann case,
the BVP has a one-dimensional nullspace.
\begin{pro}% pppppppppppppppppppppppppppppppppp
For each $(p_1,p_2)$ the solution $(\uu,p)$ to \eqref{sto}--\eqref{pr2}
is unique up to an additive constant in $p$.
\label{p:stonull}\end{pro}
\begin{proof}
The proof parallels that of Prop.~\ref{p:lapnull}.
Given any function pairs $(\uu,p)$ and $(\vv,q)$,
Green's first identity on a domain $\cal K$
is \cite[p.~53]{Ladyzhenskaya}
\be
\int_{\cal K} (\mu \Delta u_i - \partial_i p) v_i
= - \frac{\mu}{2} \int_{\cal K} (\partial_i u_j + \partial_j u_i)(\partial_i v_j + \partial_j v_i) + \int_{\partial\cal K} T_i(\uu,p) v_i
~.
\label{G1I}
\ee
Now let $(\uu,p)$ be the difference between two BVP solutions,
and set $\vv=\uu$ and ${\cal K} = \cbo$ in \eqref{G1I}.
The left-hand side vanishes due to \eqref{sto}, the $\pO$ boundary term
vanishes due to \eqref{noslip}, and the unit cell wall terms vanish by
$\uu$ periodicity, leaving only
$\int_{\cbo} \sum_{i,j=1}^2 (\partial_i u_j + \partial_j u_i)^2 = 0$.
Thus $\uu$ has zero stress, i.e.\ is a rigid motion,
so, by \eqref{noslip}, $\uu\equiv\mbf{0}$. Thus, by \eqref{sto},
and because $p$ has no pressure drop, $p$ is constant.
\end{proof}

Since, for Stokes, the Cauchy data is $(\uu,\Tv)$ \cite[Sec.~2.3]{HW},
the BVP  \eqref{sto}--\eqref{pr2} is equivalent to the BVP on a
single unit cell,
%we reframe the above as a BVP on a single unit cell,
%which, by the same reasoning as for the Laplace case (noting that
%for Stokes the Cauchy data is $(\uu,\Tv)$ \cite[Sec.~2.3]{HW}) is
%equivalent to \eqref{sto}--\eqref{pr2}:
\bea
-\mu\Delta\uu + \nabla p &=&0
\qquad \mbox{ in } \cb \backslash \overline{\Omega}
\label{sto2}
\\
\nabla\cdot\uu &=& 0
\qquad \mbox{ in } \cb \backslash \overline{\Omega}
\label{incomp2}
\\
\uu & = & \mbf{0} \qquad \mbox{ on } \pO
\label{noslip2}
\\
\uu_R - \uu_L &=& \mbf{0}
\label{uu1}
\\
\Tv(\uu,p)_R - \Tv(\uu,p)_L &=& p_1 n
\label{T1}
\\
\uu_U - \uu_D &=& \mbf{0}
\label{uu2}
\\
\Tv(\uu,p)_U - \Tv(\uu,p)_D &=& p_2 n ~,
\label{T2}
\eea
where the normal $n$ has the direction and sense for the
appropriate wall as in Fig.~\ref{f:geom}(a).
Discrepancy will refer to the stack of the four vector functions on the left-hand side of \eqref{uu1}--\eqref{T2}.

% eeeeeeeeeeeeeeeeeeeeeeeeeeeeeeeeeeeeeeeeeeeeeeeeeeeeeeeeeeeeeeeee
\subsection{The Stokes empty unit cell discrepancy BVP} % and numerical solution}
\label{s:stoempty}

Proceeding as with Laplace, one must first understand
the empty unit cell BVP with given discrepancy data
$\mbf{g}:=[\mbf{g}_1;\mbf{g}_2;\mbf{g}_3;\mbf{g}_4]$, which is
to find %the pair
a pair $(\vv,q)$ solving
\bea
-\mu\Delta\vv + \nabla q &=&0 \qquad \mbox{ in } \cb
\label{stov}
\\
\nabla\cdot\vv &=& 0
\qquad \mbox{ in } \cb
\label{incompv}
\\
\vv_R - \vv_L &=& \mbf{g}_1
\label{gg1}
\\
\Tv(\vv,q)_R - \Tv(\vv,q)_L &=& \mbf{g}_2
\label{gg2}
\\
\vv_U - \vv_D &=& \mbf{g}_3
\label{gg3}
\\
\Tv(\vv,q)_U - \Tv(\vv,q)_D &=& \mbf{g}_4 ~,
\label{gg4}
\eea
This BVP has three consistency conditions and three
nullspace dimensions, as follows.%
\footnote{Note that, although three is also the nullity of
the 2D Stokes interior Neumann BVP \cite[Table 2.3.3]{HW}, both
nullspace and consistency conditions differ from that case.}
\begin{pro} % pppppppppppppppppppppppppppppppp
A solution $(\vv,q)$ to \eqref{stov}--\eqref{gg4}
exists if and only if
\bea
\int_L \mbf{g}_2 ds + \int_D \mbf{g}_4 ds &=& \mbf{0} \qquad 
\mbox{{\rm (zero net force)}, and}
\label{gforce}
\\
\int_L n\cdot \mbf{g}_1 ds + \int_D n\cdot \mbf{g}_3 ds &=& 0 \qquad 
\mbox{{\rm (volume conservation)},}
\label{gvol}
\eea
and then is unique up to translational flow and additive pressure constants.
I.e.\ the %three-dimensional
solution space is $(\vv +\mbf{c},q + c)$
for $(\mbf{c},c) \in \RR^3$.
\label{p:stoemptynull}\end{pro}
\begin{proof}
Noting that \eqref{G1I}, with $(\uu,p)$ and $(\vv,q)$ swapped,
holds for all constant flows $\uu$ shows that $\int_{\partial \cal K} \Tv(\vv,q)=\mbf{0}$; setting $\cal K = \cb$ gives \eqref{gforce}.
\eqref{gvol} follows from the divergence theorem and \eqref{incompv}.
The proof of Prop.~\ref{p:stonull} shows that
the nullspace is no larger than constant $p$ and rigid motions for $\vv$,
but it is easy to check that rotation is excluded due to its effect on
$\mbf{g}_1$ and $\mbf{g}_3$.
\end{proof}

We solve this Stokes empty BVP in an entirely analogous fashion to
Laplace, namely via an MFS representation, but now with
vector-valued coefficients. For $\xx\in \cb$,
\bea
\vv(\xx) & \approx & \sum_{j=1}^M \phi_j(\xx) \bm\xi_j ~,
\qquad \phi_j(\xx) := G(\xx,\yy_j)~,
\qquad \yy_j := (\rp \cos 2\pi j/M, \rp \sin 2\pi j/M)~,
\label{vmfs}
\\
q(\xx) & \approx & \sum_{j=1}^M \phi^\tbox{p}_j(\xx) \cdot \bm\xi_j ~,
\qquad \phi^\tbox{p}_j(\xx) := \Gp(\xx,\yy_j)~,
\label{stomfs}
\eea
where the velocity fundamental solution (stokeslet or single-layer kernel)
is the tensor $G(\xx,\yy)$ with components
\be
G_{ij}(\xx,\yy) = \frac{1}{4\pi\mu}\left( \delta_{ij} \log \frac{1}{r}
+ \frac{r_ir_j}{r^2}\right),
\qquad i,j = 1,2, \quad \rr := \xx-\yy, \quad r:=\|\rr\|~.
\label{Gv}
\ee
and the single-layer pressure kernel is the vector $\Gp$ with components
\be
\Gp_j(\xx,\yy) = \frac{1}{2\pi}\frac{r_j}{r^2}
~,\qquad j=1,2~.
\label{Gp}
\ee
%Due to \eqref{gg2} and \eqref{gg4},
We will also need the single-layer traction kernel $\Gt$ with components
(applying \eqref{T} to the above),
\be
\Gt_{ik}(\xx,\yy) = \sigma_{ij}(G_{\cdot,k}(\cdot,\yy),\Gp_k(\cdot,\yy)) (\xx) \nx_j
= -\frac{1}{\pi}\frac{r_ir_k}{r^2}\frac{\rr\cdot\nx}{r^2}
~,\qquad i,k=1,2~,
\label{Gt}
\ee
where the target $\xx$ is assumed to be on a surface with normal $\nx$.
%analog of the ``derivative of SLP'' (or transpose of DLP) for Laplace.

The generalization of the linear system from (scalar) Laplace to (vector)
Stokes is routine bookkeeping, which we now outline.
The MFS coefficient vector $\xi := \{ \bm\xi_j\}_{j=1}^M$ has $2M$ unknowns,
which we order with the $M$ 1-components followed by the $M$ 2-components.
% * thm on stokes mfs converngece? ref or discuss future?
Discretizing \eqref{gg1}--\eqref{gg4} on $m$ collocation nodes per wall,
as in Sec.~\ref{s:lapempty}, gives $Q\xi = \mbf{g}$ as in \eqref{Qsys},
with discrepancy vector $\mbf{g} \in \RR^{8m}$.
%being a stack of four vector-valued samples on $m$ nodes.
We choose to order each of the four blocks in $\mbf{g}$ with
all $m$ 1-components followed by $m$ 2-components.
As before, $Q=[Q_1;Q_2;Q_3;Q_4]$ with each $Q_k$ split into $2\times 2$
sub-blocks based on the 1- and 2-components.
For instance,
writing $\phi_j^{kl}$, $k,l=1,2$, for the four components of the
basis function $\phi_j$ in \eqref{vmfs}, then the $R$-$L$ velocity block
$Q_1 = \mt{Q_1^{11}}{Q_1^{21}}{Q_1^{12}}{Q_1^{22}}$,
%\left[\begin{array}{ll}Q_1^{11} & Q_1^{21} \\ Q_1^{12} & Q_1^{22} \end{array}\right]$
with each sub-block having entries
$(Q_1^{kl})_{ij} = \phi_j^{kl}(\xx_{iL} +\ex) - \phi_j^{kl}(\xx_{iL})$.
The $R$-$L$ traction block $Q_2$ has sub-blocks
$(Q_2^{kl})_{ij} = \Gt_{kl}(\xx_{iL} +\ex,\yy_j) - \Gt_{kl}(\xx_{iL},\yy_j)$,
where it is implied that the target normal is on $L$.
% easier than taking T(phi,varphi) which is a mess
The other two blocks are similar.

In Fig.~\ref{f:Qconv}(c) we show the convergence of the
solution error produced by numerical solution of
the Stokes version of \eqref{Qsys} (again, see Remark~\ref{r:linsolve}).
The results are almost identical to the Laplace case in Fig.~\ref{f:Qconv}(b),
converging to 15 digit accuracy
with a similar exponential rate; keep in mind that now there are $2M$ unknowns
rather than $M$.
Note that we do not know of a Stokes version of Theorem~\ref{t:Qconv}.
Fig.~\ref{f:Qconv}(c) also shows that three singular values decay faster
than the others, as expected from Prop.~\ref{p:stonull}.

Finally, the
discretization of the consistency conditions in Prop.~\ref{p:stonull}
is the statement that
\be
W\tr Q \; \approx \; \mbf{0}_{3\times 2M}~,
\label{WQ}
\ee
where $\mbf{0}_{m\times n}$ is the $m$-by-$n$ zero matrix,
and the weight matrix $W$ (the analog of \eqref{w}) is
\be
W = \left[\begin{array}{llllllll}
0&0&\ww_L\tr&0&0&0&\ww_D\tr&0\\
0&0&0&\ww_L\tr&0&0&0&\ww_L\tr\\
\ww_L\tr&0&0&0&0&\ww_D\tr&0&0
\end{array}\right]\tr
\; \in \; \RR^{8m\times 3}
~.
\label{W}
\ee

\bfi % fffffffffffffffffffffffffffffffffffffffffffffffffff
\raisebox{-2.6in}{\ig{width=3.1in}{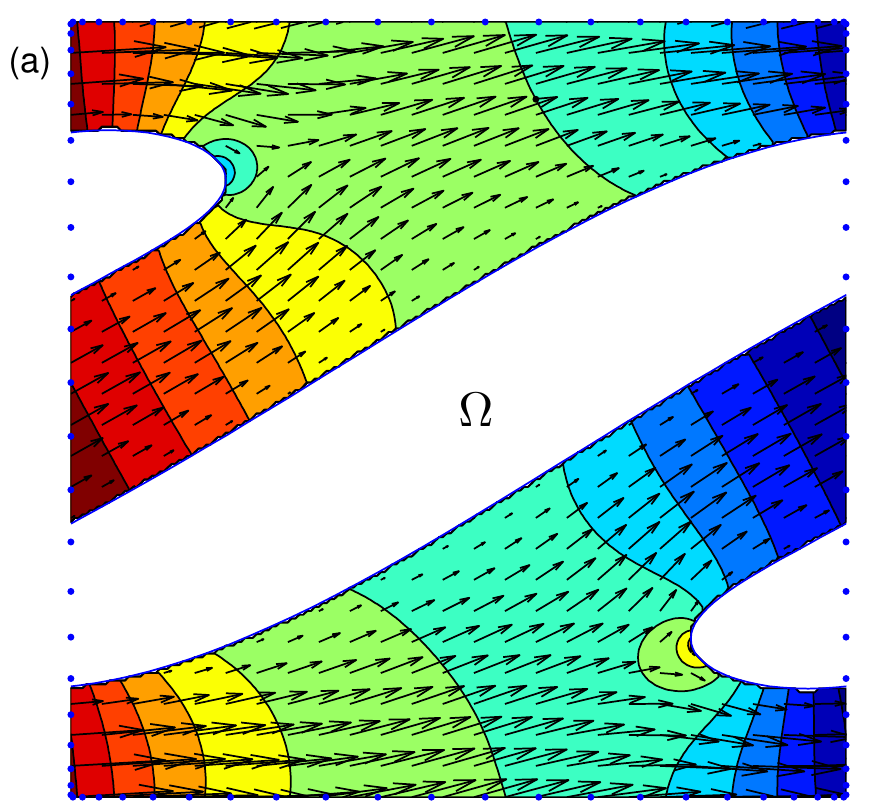}}
\quad
(b)
\raisebox{-2.6in}{\ig{width=2.8in}{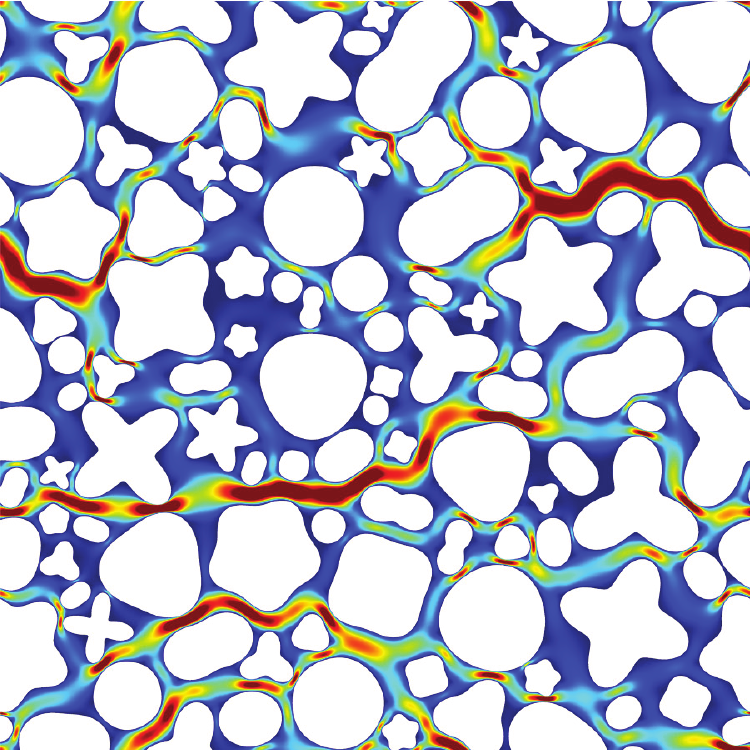}}
\\
\ig{width=2.2in}{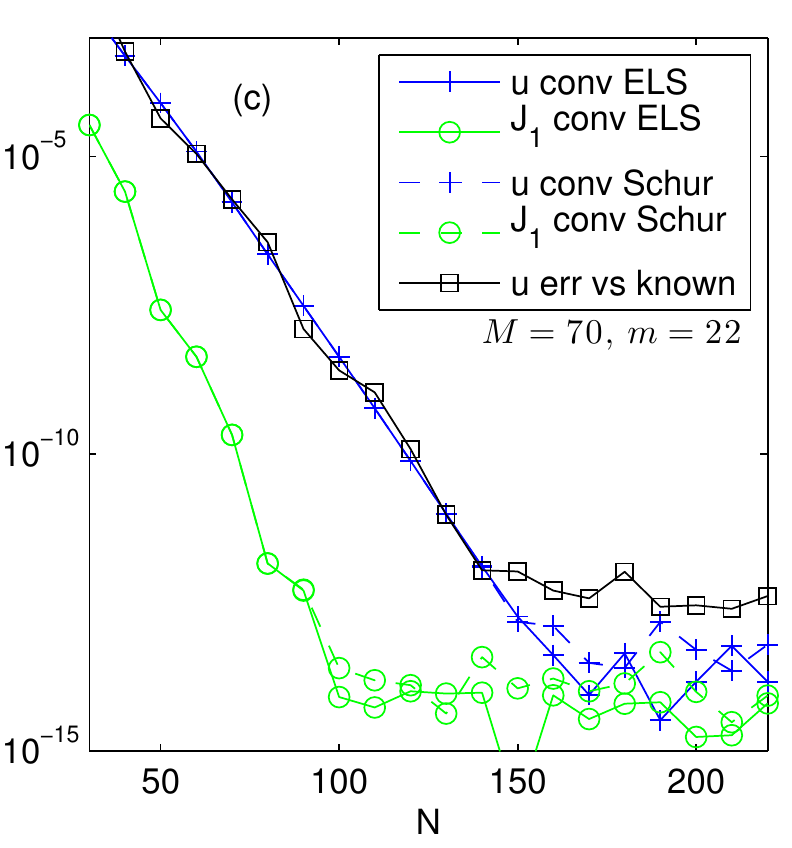}
\hspace{0.95in}
%(d) $K=10^2$ $N_k$-conv plot
\ig{width=2.2in}{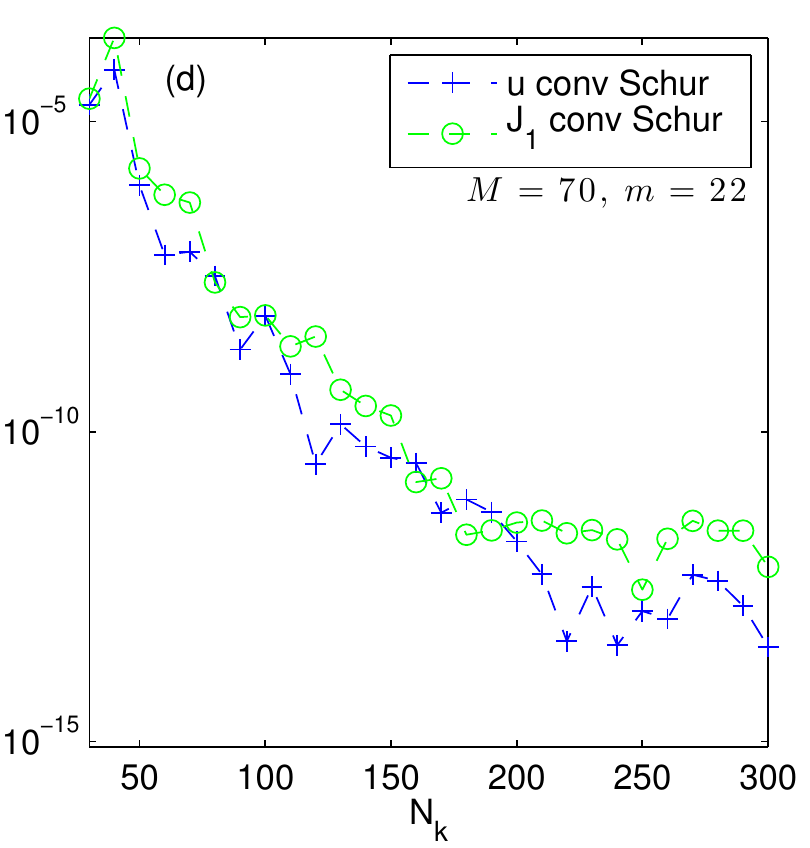}
%\quad \ig{width=2in}{stoMconv.eps}  % skip it
\ca{Periodic Stokes flow tests with applied pressure drop $\mbf{p} = (1,0)$.
(a) Solution velocity $\uu$ (arrows) and pressure $p$ (contours)
for Example 4 (same geometry as Fig.~\ref{f:ELS}(a)),
with $N=150$, $M=80$, and $\rp=1.4$.
A single unit cell is shown, and the $m=22$ nodes per wall are shown as dots.
(b) Flow speed $|\uu|$ for $K=100$ inclusions per unit cell,
using iterative solution of \eqref{pschur2}
(Example 6; geometry identical to Fig.~\ref{f:ELS}(b)).
(c) $N$-convergence of $\uu$ error (with $M=70$ fixed) at the same test points
as in Fig.~\ref{f:ELS}(a),
and of flux $J_1$ (see end of Sec.~\ref{s:stobie}),
relative to their values at $N=230$ (Example 4).
Squares show error convergence in $\uu$
in the case of known $\uu_\tbox{ex}$ due to a stokeslet grid.
Solid lines are for direct solution of the ELS, 
dashed lines for the iterative solution of \eqref{pschur2}.
(d) Convergence with $N_k$ for $K=100$ inclusions per unit cell (Example 6).
Errors in velocity (+ signs) and flux $J_1$ (circles) are shown,
estimated relative to their values at $N_k=400$.
% skip any known soln here.
%
%(e) $M$-convergence of error in $J_1$ at the fixed $N$ and $m$ shown,
%for the above cases $K=1$ (+ signs) and $K=10^2$ (squares) ***.
%For $K=1$ the lowest six singular values of $E$ are also shown (grey lines),
%and the solution norm $\|[\btau;\xi]\|_2$ (points).
}{f:sto}
\efi

% bbbbbbbbbbbbbbbbbbbbbbbbbbbbbbbbbbbbbbbbbbbbbbbbbbbbbbbbbbbbbbbbbbbbbbbbb
\subsection{Integral representation and the Stokes ELS}
\label{s:stobie}

In contrast to the Neumann Laplace problem, we are now solving
a Dirichlet BVP, which suggests a pure double-layer representation
on inclusion boundaries.
However, this would lead to a 3D nullspace for each inclusion,
associated with its complementary interior Neumann BVP \cite[Table~2.3.3]{HW}.
We remove this nullspace via a ``compound'' or mixed
double-layer formulation, namely an admixture of double- and single-layers
\cite[Thm.~2.1]{hebeker} \cite[p.128]{pozrikidis};
this avoids the use of extra interior stokeslet
degrees of freedom which can add $O(K^3)$ work for $K$ inclusions
\cite{Krop04}.

For convenience we define the 2D Stokes layer potentials.
Given a vector-valued density $\btau = (\tau_1,\tau_2)$ on $\pO$,
the Stokes single-layer potential generates (via \eqref{Gv} and \eqref{Gp})
velocity and pressure
\be
\uu(\xx) = ({\cal S}_\Gamma \bm\tau)(\xx) :=
\int_\Gamma G(\xx,\yy) \bm\tau(\yy) ds_\yy~,
\qquad
p(\xx) = ({\cal S}^\tbox{p}_\Gamma \bm\tau)(\xx) :=
\int_\Gamma \Gp(\xx,\yy) \bm\tau(\yy) ds_\yy~.
\label{uupSLP}
\ee
Using $\ny$ to denote the source normal, the Stokes double-layer velocity
potential is
\be
\uu(\xx) = ({\cal D}_\Gamma \bm\tau)(\xx) :=
\int_\Gamma D(\xx,\yy) \bm\tau(\yy) ds_\yy
\quad \mbox{ where }
D_{ij}(\xx,\yy) = \frac{1}{\pi}\frac{r_ir_j}{r^2}\frac{\rr\cdot\ny}{r^2}
~,\qquad i,j=1,2~,
\label{D}
\ee
a kernel which is the negative transpose of \eqref{Gt}, and the 
associated pressure
\be
p(\xx) = ({\cal D}^\tbox{p}_\Gamma \bm\tau)(\xx) :=
\int_\Gamma \Dp(\xx,\yy) \bm\tau(\yy) ds_\yy
\quad \mbox{ where }
D^\tbox{p}_{j}(\xx,\yy) = \frac{\mu}{\pi}\left( -\frac{\ny_j}{r^2} + 2 \rr\cdot\ny\frac{r_j}{r^4}\right),
\quad i,j=1,2.
\label{Dp}
\ee
The Stokes analog of the boundary integral operator
\eqref{DT} %, which we also call $D\tr$
is the single-layer traction
\be
(D\tr_{\Gamma',\Gamma} \btau)(\xx) :=
\int_\Gamma \Gt(\xx,\yy) \bm\tau(\yy) ds_\yy~, \qquad \xx\in\Gamma'~,
\label{DTs}
\ee
which uses \eqref{Gt}.
Finally, the hypersingular double-layer traction operator is needed,
\be
(T_{\Gamma',\Gamma} \btau)(\xx) :=
\int_\Gamma \Dt(\xx,\yy) \bm\tau(\yy) ds_\yy~, \qquad \xx\in\Gamma'~,
\label{Ts}
\ee
whose kernel is computed by inserting \eqref{D} and \eqref{Dp}
into \eqref{T} and simplifying to get (e.g.\ \cite[(5.27)]{Liu09book}),
\be
\Dt_{ik} = 
\frac{\mu}{\pi}\left[
\left(\frac{\ny\cdot\nx}{r^2}- 8\dx\dy\right)\frac{r_ir_k}{r^2}
+\dx\dy\delta_{ik} +\frac{\nx_i\ny_k}{r^2}
+\dx\frac{r_k \ny_i}{r^2} + \dy \frac{r_i \nx_k}{r^2}
\right]
\label{Dt}
\ee
where for conciseness we defined the target and source ``dipole functions'',
respectively
$$
\dx = \dx(\xx,\yy) := (\rr\cdot\nx)/r^2
~,\hspace{1in}
\dy = \dy(\xx,\yy) := (\rr\cdot\ny)/r^2 ~.
$$
% oops note x y swap error in prev work! (periodicp)

The Stokes jump relations \cite[Sec.~3.2]{Ladyzhenskaya}
are identical to the usual Laplace ones \eqref{JR}
%(and \cite[Sec.~6.3]{LIE}),
with the potential taken to be velocity potential,
and the normal derivative replaced by the traction.
In short, the single-layer velocity and double-layer tractions are continuous,
whereas the single-layer traction jump is minus the density,
and the double-layer velocity jump is the density itself.

Armed with the above,
our representation for the solution $(\uu,p)$
in $\cbo$ is the mixed formulation,
\be
\uu \; = \; ({\cal D}^\tbox{near}_\pO+{\cal S}^\tbox{near}_\pO)\btau +
\sum_{j=1}^M \phi_j \bm\xi_j~,
\qquad
p \; = \; ({\cal D}^\tbox{p,near}_\pO+{\cal S}^\tbox{p,near}_\pO)\btau +
\sum_{j=1}^M \phi^\tbox{p}_j \cdot \bm\xi_j~,
\label{uurep}
\ee
where ``near'' denotes $3\times 3$ image sums as in \eqref{urep}.
With this choice made, the continuous
%integral-algebraic  ? sounds scary
form of the ELS is similar to \eqref{bie}--\eqref{row2d}.
The first block is, applying the exterior jump relation to 
\eqref{noslip2},
%\cite[Sec.~3.2]{Ladyzhenskaya}
\be
(-\half + D^{\tbox{near}}_{\pO,\pO}+S^{\tbox{near}}_{\pO,\pO})\btau + \sum_{j=1}^M \phi_j|_\pO \bm\xi_j \;= \;\mbf{0}~.
\label{stobie}
\ee
Rather than list all second subblocks, we note that they are
simply \eqref{row2a}--\eqref{row2d} with
$S$ replaced by $D+S$, $D\tr$ replaced by $T+D\tr$, and
$\partial \phi_j(\xx)/\partial n$ replaced by the basis traction,
which has kernel $\Gt_{lk}(\xx,\yy_j)$.
Finally, the pressure driving $(p_1,p_2)$ is encoded by the discrete
right hand side
\be
\mbf{g} \;=\; [0;0;p_1\mbf{1}_m;0;0;0;0;p_2\mbf{1}_m] \;\in\; \RR^{8m}
~,
\label{stogg}
\ee
so that the ELS has, as for Laplace, the form \eqref{ELS}.

This mixed formulation allows an analog of
Lemma~\ref{l:laptau}: the Stokes ELS nullspace is spanned by $\bm\xi$.
\begin{lem}  % lllllllllllllllllllllllllllllllllllllllllllllllll
In the solution to the Stokes ELS (i.e.\ \eqref{stobie} plus the Stokes
version of \eqref{row2a}--\eqref{row2d} described above), $\btau$ is unique.
%$\sum_{j=1}^M \phi_j \bm\xi_j$ is unique, and $\sum_{j=1}^M \phi^\tbox{p}_j \cdot \bm\xi_j$ is unique up to a constant (pressure) function.
\label{l:stotau}\end{lem}
\begin{proof}
The proof parallels that of Lemma~\ref{l:laptau}, but adapts the 2D
version of the proof of
\cite[Thm.~2.1]{hebeker} for the interior uniqueness step.
Let $\vv^+$ and $\vv^-$ denote
the exterior and interior limits respectively of $\vv$ on $\pO$,
and $\mbf{T}^+$ and $\mbf{T}^-$ be the same limits of $\mbf{T}(\vv,q)$.
If $(\btau,\bm\xi)$ is the difference between any two ELS solutions,
let $(\vv,q)$ be given by the representation \eqref{uurep}
both in $\cbo$ and inside $\Omega$.
Thus, by Prop.~\ref{p:stonull}, $\vv\equiv\mbf{0}$ and $q=c$, some constant,
and $\mbf{T}^+=-cn$.
The jump relations for ${\cal D}+{\cal S}$ then give
$\vv^- = \btau$ and $\mbf{T}^- = -cn - \btau$.
Thus in $\Omega$, $(\vv,q)$ is a Stokes solution with Robin (impedance)
data, $\vv^- + \mbf{T}^-= -cn$.
Applying \eqref{G1I} with $\uu=\vv$ and ${\cal K} = \Omega$ gives
$0 = -(\mu/2)\int_\Omega (\partial_i v_j+\partial_j v_i)^2
+ c\int_\pO n\cdot \vv^- - \int_\pO \|\vv^-\|^2$.
However by incompressibility $\int_\pO n\cdot \vv^- = 0$,
and the two remaining non-positive terms have the same sign,
so must both vanish.
Hence $\vv\equiv\mbf{0}$ in $\Omega$.
By the double-layer jump relation, $\btau = \vv^+-\vv^-$,
both of which have been shown to vanish, thus $\btau\equiv\mbf{0}$.
%The statements about $\bm\xi$ then follow from \eqref{uurep}.
\end{proof}

The filling of the discrete $(2N+8m)$-by-$(2N+2M)$
ELS is now routine, being as in Sec.~\ref{s:lapbie},
apart from vector bookkeeping, and the following two details.
i) For the discretization of $D_{\pO,\pO}$ with kernel \eqref{D},
the diagonal limit at node $\xx_k$ is
$D_{ij}(\xx_k,\xx_k) = (-\kappa(\xx_k)/2\pi)t_i(\xx_k)t_j(\xx_k)$
where $(t_1,t_2)$ are the components of the unit tangent vector
on $\pO$.
ii) The kernel of $S_{\pO,\pO}$ has a logarithmic singularity on the diagonal,
so the plain Nystr\"om rule cannot be used.
However, there exist several high-order discretizations for such a kernel
\cite{LIE,alpert,hao};
when the $N$ per inclusion is not large
we prefer the spectral product quadrature due to Martensen, Kussmaul and Kress
\cite[Ch.~12.3]{LIE} \cite[Sec.~6.2]{hao}.

We provide an implementation of
the above kernel quadratures in \verb?StoSLP.m? and \verb?StoDLP.m?, and
of filling the ELS in \verb?fig_stoconvK1.m?;
see Remark~\ref{r:code}.

{\bf Example 4.}
We illustrate the above with results
obtained via a simple dense direct solve of the ELS, for the periodic ``worm''
geometry from Example 1, with no-slip boundary conditions,
pressure driving $\mbf{p}=(1,0)$, and $\mu=0.7$.
Fig.~\ref{f:sto}(a) shows the resulting flow $\uu$ and pressure $p$.
Here we evaluated $\uu$ using the close-evaluation
quadratures from \cite{lsc2d}, and $p$ using similar formulae,
to provide spectral accuracy even up to the curve.
As expected from Prop.~\ref{p:stonull}, the ELS exhibits unit
numerical nullity: the Stokes version of Fig.~\ref{f:ELS}(e) is very similar,
and we do not show it.
Fig.~\ref{f:sto}(c) shows the pointwise convergence in $\uu$,
and is consistent with exponential convergence
down to 13 digit accuracy.
For a known solution produced by stokeslets on the same grid as
in Sec.~\ref{s:lapnum}, convergence stops at around 12 digits.

Rapid convergence of the flux $J_1$ is also shown in Fig.~\ref{f:sto}(c),
reaching 13 digit accuracy at only $N=100$. Here, flux is evaluated
using the Stokes version of \eqref{J1trick}, which replaces
the representation ${\cal S}$ by ${\cal D} + {\cal S}$,
replaces $\partial/\partial n$ by $n \cdot$, and in the proof
replaces the zero-flux condition by incompressibility.

% cccccccccccccccccccccccccccccccccccccccccccccccccccccccccccccccccccccccccc
\subsection{Schur complement systems for Stokes}
\label{s:stosch}

Having filled the Stokes version of the ELS \eqref{ELS},
we would like to eliminate the periodizing unknowns $\bm\xi_j$ to leave
a $2N\times 2N$ system that can be solved iteratively.

\subsubsection{An equivalent square system preserving the nullspace}  % ssssssss
\label{s:stoequiv}
It is possible to build such a system that preserves the rank-3 nullity
using exactly the recipe in Section~\ref{s:equiv}.
For this we need only two ingredients: i) $H$ must be a stack of three
densities so that $CH$ enlarges the range of $\tilde Q$
to include all smooth vectors, and ii)
$R$ must be a stack of three proxy coefficient vectors with
full rank projection onto $\Nul Q$.
%We now construct these.

We first need the Stokes version of the technical lemma.
For Stokes there are {\em two} types of Gauss' law.
%the first for force and the second for volume.
Let $\vv = {\cal S}_\pO\bm\sigma + {\cal D}_\pO\btau$ be a velocity
potential with arbitrary densities $\bm\sigma$ and $\btau$ on a curve $\pO$,
with associated pressure
$q = {\cal S}^\tbox{p}_\pO\bm\sigma + {\cal D}^\tbox{p}_\pO\btau$.
Then,
\be
\int_{\partial \cal K} \Tv(\vv,q) = \int_\pO \bm\sigma\,ds
\qquad\mbox{(net force)},
\quad\mbox{and}\qquad
\int_{\partial \cal K} \vv\cdot n = \int_\pO \btau \cdot n \,ds
\qquad\mbox{(fluid volume)},
\label{stoGau}
\ee
where $\partial \cal K$ is the boundary of some open domain $\cal K$ containing $\pO$.
The proofs of both are similar to that of \eqref{Gauss},
summing over the boundaries of $\Omega$ and of ${\cal K}\backslash\overline{\Omega}$
to get the left side integral.
For the first law one uses the single-layer traction jump relation,
and, for the second, the double-layer velocity jump relation
\cite[p.~57--58]{Ladyzhenskaya}.
\begin{rmk}
The second law of \eqref{stoGau} states that the Stokes double-layer
velocity potential does not generally conserve fluid volume,
%is not divergence-free on $\pO$,
%does not in general satisfy $\nabla \cdot \vv=0$ on $\pO$,
%i.e.\ it does not conserve fluid volume,
a surprising fact that we did not find stated in standard literature.
\end{rmk}

Now, for $H$ we can choose densities that, when used with our $\cal D + S$
representation, generate net force in the 1-direction,
in the 2-direction, and break volume conservation, respectively:
\be
H = \left[\hhh{1}\; \hhh{2}\; \hhh{3}\right] = \left[\begin{array}{lll}
\mbf{1}_N & \mbf{0}_N & \{n^{\mbf{x}_i}_1\}_{i=1}^N \\
\mbf{0}_N & \mbf{1}_N & \{n^{\mbf{x}_i}_2\}_{i=1}^N
\end{array}\right]~,
\label{H}
\ee
where in the columns of the matrix
we use ``nodes fast, components slow'' ordering,
and recall that the normals live on the nodes of $\pO$.
Next we show that $CH$ has full-rank projection onto $\Ran W$.
%This $H$ satisfies the analog of Lemma~\ref{l:onesdens}.
\begin{lem}
Given $H$ in \eqref{H}, $W$ in \eqref{W}, and $C$ the Stokes discrepancy
matrix filled as in Sec.~\ref{s:stobie}, then
the $3\times 3$ matrix $W\tr C H$ is diagonal, and,
for all $N$ and $m$ sufficiently large, invertible.
\label{l:onesdens2}
\end{lem}
\begin{proof}
The proof parallels that of Lemma~\ref{l:onesdens}.
%Letting $\mbf{d}=[\mbf{d}_1;\mbf{d}_2;\mbf{d}_3;\mbf{d}_4]$ be the
For instance,
setting $\vv = ({\cal S}^\tbox{near}_\pO + {\cal D}^\tbox{near}_\pO)\hhh{1}$
and using \eqref{stoGau}
we see the discrepancy has net force $|\pO|$ in the 1-direction only,
and since $\int_\pO n = \mbf{0}$ the discrepancy is volume-conserving.
Discretizing shows that the first column of $W\tr C H$ converges
(in $m$ and $N$) to $[|\pO|\;0\;0]\tr$.
Similar steps apply to the other two columns, showing that the
matrix in fact converges to $|\pO|$ times the $3\times 3$ identity.
% eg using that $\int_\pO n\cdot n = |\pO|$
\end{proof}

For $R$ we choose the three analogous scaled
coefficient vectors on the proxy circle:
\be
R = \left[\rrr{1}\; \rrr{2}\; \rrr{3}\right] = \frac{1}{M}
\left[\begin{array}{lll}
\mbf{1}_M & \mbf{0}_M & \{\cos 2\pi j/M \}_{j=1}^M \\
\mbf{0}_M & \mbf{1}_M & \{\sin 2\pi j/M \}_{j=1}^M
\end{array}\right]~.
\label{R}
\ee
Reasoning as in the start of the proof of Theorem~\ref{t:gary},
the three (linearly independent) columns of $R$ are
exponentially close to generating
constant flow in the 1-direction, the 2-direction, and constant pressure,
respectively, which, according to
Prop.~\ref{p:stoemptynull}, fall into $\Nul Q$.
%Thus $R$ has full rank projection onto $\Nul Q$.

Given the above statements about $H$ and $R$,
Theorem~\ref{t:gary} holds
for the Stokes version of the procedure of Section~\ref{s:equiv};
in its proof one need only change $\ww$ to the $W$ of \eqref{W}.
Because of the unit nullity of the Stokes ELS,
this would result in a rank-1 deficient square system matrix. % $A_\tbox{per}$.

\subsubsection{A well-conditioned square Stokes system} % sssssssssssssssssssssssssssss
\label{s:stowellcond}
To remove this nullspace, it is tempting to take the obvious Stokes
generalization of Section~\ref{s:wellcond}, leaving the $B$ block
unperturbed as in \eqref{pELS}. However, this fails (the solution has constant
but incorrect pressure drops) because
$W\tr \mbf{g} = [p_1;p_2;0] \neq \mbf{0}_3$.
%breaking the proof of
%a Stokes version of Theorem~\ref{t:alex}.
However, since only one nullspace dimension needs be removed,
and this nullspace is common to the empty unit cell BVP, this suggests
a ``hybrid'' where the perturbation of $B$ involves {\em only the remaining two}
dimensions.
This indeed leads to a well-conditioned square system:
\ben
\item Set $H$ as in \eqref{H} and $R$ as in \eqref{R}, then set
$\tilde Q := Q + C H R\tr$ and $\hat B:= B + A P_{12}$, where
$P_{12} := [\hhh{1}\;\hhh{2}] [\rrr{1}\;\rrr{2}]\tr$
involves the two force (but not the one pressure) perturbations.
\item Solve for the vector $\hat\btau$ in the $2N\times 2N$ Schur
complement linear system
\be
\hat A_\tbox{per} \hat\btau := (A - \hat B \tilde Q^+ C)\hat\btau \; = \;
-\hat B \tilde Q^+ \mbf{g}~,
\label{pschur2}
\ee
%which may be done iteratively.
where, as before, one solves the small systems
$\tilde Q X = C$ and $\tilde Q \mbf{y} = \mbf{g}$,
to build the large system matrix $\hat A_\tbox{per} = A - \hat B X$
and right-hand side
$-\hat B \mbf{y}$ for \eqref{pschur2},
which may then be solved iteratively. %(see Remark~\ref{r:fmm}).
%Within a Krylov subspace iterative solver,
\item Recover the proxy coefficients via $\bm\xi = \mbf{y} - X \hat\btau$.
\item Recover the density via
$\bm\tau = \hat\btau + P_{12}\bm\xi$.
\een
\begin{thm} % ttttttttttttttttttttttttttttttttttttttttttttttttttttttttttttttt
Let $\mbf{g}$ encode the pressure driving as in \eqref{stogg}.
Then for all $N$, $M$ and $m$ sufficiently large, the pair
$(\btau,\bm\xi)$ produced by the above procedure is unique,
and solves the Stokes version of the ELS \eqref{ELS}
with residual of order the quadrature error on boundaries.
\label{t:stoalex}\end{thm}            % tttttttttttttttttt

This is a Stokes analog of Theorem~\ref{t:alex}, and its proof structure
is similar.
\begin{proof}
\eqref{pschur2} is the Schur complement of the perturbed ELS
\be
\mt{A}{B+AP_{12}}{C}{Q+CHR\tr}\vt{\hat\btau}{\bm\xi} \; = \; \vt{0}{\mbf{g}}
~,
\label{pELS2}
\ee
so that if $\hat\btau$ solves \eqref{pschur2} then $(\hat\btau,\bm\xi)$,
with $\bm\xi$ as in Step 3, solves \eqref{pELS2}.
Applying Step 4, one may check that the pair $(\btau,\bm\xi)$ then solves
a rank-1 perturbed system similar to \eqref{pELS},
\be
\mt{A}{B}{C}{Q+C\hhh{3}\rrr{3}\tr}\vt{\btau}{\bm\xi} \; = \; \vt{0}{\mbf{g}}
~.
\label{pELS2t}
\ee
From this pair, define $\vv$ as the resulting quadrature approximation to the
velocity potential \eqref{uurep}, namely
$$
\vv(\xx) := 
\sum_{m,n\in\{-1,0,1\}} \sum_{i=1}^N w_i
[D(\xx,\xx_i+m\ex+n\ey) + G(\xx,\xx_i+m\ex+n\ey)] \btau_i
+ \sum_{j=1}^M \phi_j(\xx) \bm\xi_j~.
$$
Since it matches that of the ELS, the first row of \eqref{pELS2t} implies that
$\vv|_\pO=\mbf{0}$.
Thus $\int_\pO \vv \cdot n = 0$ and, by volume conservation in $\cbo$,
the discrepancy of $(\vv,q)$ obeys \eqref{gvol}, whose discretization
is $\www{3}\tr (C \btau + Q \bm\xi) = 0$, where $\www{3}$ is the third column
of \eqref{W}.
Subtracting this from the 2nd row of \eqref{pELS2t} 
left-multiplied by $\www{3}\tr$ leaves
$(\www{3}\tr C \hhh{3}) \rrr{3}\tr \bm\xi = \www{3}\tr \mbf{g}$.
Lemma~\ref{l:onesdens2} implies that $\www{3}\tr C \hhh{3}\neq0$,
and by computation $\www{3}\tr \mbf{g}=0$.
Thus $\rrr{3}\tr \bm\xi =0$, so the pair also solves the ELS.
Thus Lemma~\ref{l:stotau} holds, so $\btau$ is unique.
\end{proof}

We expect \eqref{pschur2} to be well-conditioned if $A$ is,
since the unit nullity has been removed be enforcing one extra condition
$\rrr{3}\tr \bm\xi =0$.  In contrast to Laplace, here the rank-3 nullity
of the empty BVP demands a hybrid scheme which ``segregates''
the other two conditions.

We test this scheme with Example 4 (Fig.~\ref{f:sto}).
Panel (c) includes (as dotted lines)
the convergence using the above recipe,
using GMRES with tolerance $10^{-14}$.
The convergence is almost identical to that of the direct ELS solution,
and the converged values differ by around $2\times 10^{-14}$.
The condition number of $\hat A_\tbox{per}$ is 82.7,
independent of $N$ and $M$,
and the number of GMRES iterations varied between 25 and 30.

\begin{table} % ttttttttttttttttttttttttttttttttttttttttttttttttttttttttttt
\begin{tabular}{lllllll}      % Greengard-Kropinski 04:
$c$ (vol.\ frac.) & $N$ & \# iters.\ & CPU time & est.\ rel.\ err.\
& $D_\tbox{calc}$ & $D_\tbox{GK}$ \\
\hline 
0.05 & 60 & 5 & 0.06 s & $6\times 10^{-15}$ & 15.557822318902 & 15.5578 \\
0.1 & 60 & 7 &  0.03 s & $2\times 10^{-14}$ & 24.831675248849 & 24.8317 \\
0.2 & 60 & 8 & 0.03 s  & $3\times 10^{-14}$ & 51.526948961799 & 51.5269 \\
0.3 & 60 & 10 & 0.03 s & $6\times 10^{-14}$ & 102.88130367004 & 102.881 \\
0.4 & 60 & 11 & 0.04 s & $1\times 10^{-14}$ & 217.89431472214 & 217.894 \\
0.5 & 100 & 11 &0.05 s & $2\times 10^{-13}$ & $5.325481184629\times 10^2$ & $5.32548 \times 10^2$\\
0.6 & 150 & 15 & 0.06 s & $2\times 10^{-12}$ & $1.76357311252\times 10^3$ & $ 1.76357 \times 10^3$\\
0.7 & 350 & 21 & 0.2 s & $2\times 10^{-11}$ & $1.35191501296\times 10^4$ &  $1.35193 \times 10^4$\\
0.75 & 800 & 30 & 0.9 s & $5\times 10^{-11}$ & $1.2753159106\times 10^5$ & $1.27543 \times 10^5$\\
0.76 & 1200 & 32 & 1.9 s & $1\times 10^{-9}$ & $2.948165198\times 10^5$ & $2.94878 \times 10^5$\\
0.77 & 2400 & 39 & 6.5 s & $3\times 10^{-9}$ & $1.038269713\times 10^6$ & $1.03903 \times 10^6$\\
\hline
\end{tabular}
\vspace{0.5ex}
\ca{Computation of the 2D viscous Stokes drag
of an infinite square array of no-slip
discs of various volume fractions $c$ (Example 5).
The well-conditioned Schur system of \eqref{pschur2} was used.
$D_\tbox{calc}$ shows the dimensionless drag results.
Relative errors were estimated by convergence
(comparing against $N$ values up to 400 larger than shown, and increasing $M$).
The uncertainty is at the level of the last digit
quoted (all preceding digits are believed correct).
The last column $D_\tbox{GK}$ shows for comparison the numerical results
of \cite[Table~1, 5th column]{Krop04}.
}{t:discarray}
\end{table} %% ttttttttttttttttttttttttttttttttttttttttttttttttttttt

\subsection{More challenging numerical tests for Stokes}
\label{s:stonum}

We now validate the above periodic Stokes scheme against a known test case,
and in geometries with large numbers of inclusions.

{\bf Example 5.}
The effective permeability of the no-slip flow around an infinite square
lattice of discs is a standard test case.
We computed this for the list of volume fractions $c$ tested
by Greengard--Kropinski \cite{Krop04}.
We used $N$ uniform quadrature points on the disc boundary,
chosen such that 
the separation between discs was always at least six times the
local quadrature node spacing $h$, so that the native Nystr\"om
quadrature was accurate;
with this choice
our $N$ values come out similar to those in \cite[Table~1]{Krop04}.
%no close-evaluation quadratures were needed - they made it worse
As $c$ approaches $\pi/4$ (the maximum possible volume fraction), the
discs come closer, the solution density becomes more peaked,
and the required $N$ diverges.
Only one pressure drop $\mbf{p} = (1,0)$ need be solved, since the
permeability tensor is a multiple of the identity.
The dimensionless drag is then related to the permeability $\kappa_{11}$ by
$D = 1/(\mu \kappa_{11})$.
We used GMRES with dense matrix-vector multiplication, and find
similar numbers of iterations as \cite{Krop04}.
The GMRES tolerance was set to $10^{-14}$, although sometimes the algorithm
stagnated at a relative residual one digit worse than this.
CPU times are quoted for a laptop with an i7-3720QM processor at 2.6 GHz,
running the code \verb?tbl_discarray_drag.m? (see Remark~\ref{r:code}).

Our results are given in Table~\ref{t:discarray}.
To validate the correctness of our scheme we also include
in the last column the numerical results from \cite[Table~1]{Krop04},
which were quoted to 6 digits, and in turn were validated against
other published work.
For $c\le 0.6$,
our results match the 6 digits quoted in \cite[Table~1]{Krop04},
but above this the match deteriorates,
reaching only 3 digits at the most challenging $c=0.77$.
By testing convergence % at higher $N$
we believe that we achieve 14 digits of accuracy
for the smaller $c$ values, dropping to around
9 digits at the larger ones.
%In all cases we achieve 6--8 digits more than \cite{Krop04}.
For $c=0.77$, where the gap between discs is only $10^{-2}$
(i.e.\ $f_\tbox{clup} \approx 638$),
% 2*pi/(1-sqrt(0.77/(pi/4)))
we observe relative fluctuation at the $10^{-9}$ level
(i.e.\ the last digit quoted in the last entry of the $D_\tbox{calc}$ column)
even when increasing $N$ to much larger values.
%This may be associated with the cond(Aper) of 2e3, set of sing vals getting small
Further work is needed to ascertain whether this limit is
controlled by the underlying condition number of the BVP
(and hence impossible to improve upon in double-precision arithmetic),
or, conversely, if adaptive panel-based close-evaluation quadratures
\cite{helsingtut} could reduce the error.

{\bf Example 6.}
We solve the geometry of Example 2, with $K=100$ inclusions,
but with Stokes no-slip conditions, and the
well-conditioned iterative scheme of Sec.~\ref{s:stowellcond}.
We apply Stokes potentials using a few calls to the Laplace FMM,
as presented in \cite[Sec.~2.1]{lsc2d}.
Close-evaluation quadratures for $S$ and $D$ are
used, as in Example 4.
One period of the resulting
solution flow speed is plotted in Fig.~\ref{f:sto}(b).
Convergence is shown by Fig.~\ref{f:sto}(d); 11-12 digits
of accuracy in flux is achieved,
which is around 2 digits less than for Laplace.
At the converged value $N_k=200$, the solution takes
1720 seconds.
% *** how many GMRES iters?

Linear complexity as a function of $K$ is demonstrated by
Fig.~\ref{f:GMRES}(c). The time per iteration is similar to that
for Laplace; this is because we pre-store the close-evaluation matrix blocks
whereas for Laplace we evaluated them on the fly.
We show the solution for $K=10^3$ in Fig.~\ref{f:stocandy},
using $N_k=350$.
This required 1289 GMRES iterations and took 21 hours of CPU time.
By comparison to the solution at $N_k=450$, we estimate that the
absolute flux error is $9\times 10^{-9}$.
However, due to the narrow channels in this porous medium,
the size of the flux is only around $10^{-5}$,
meaning that {\em relative} accuracy is only 3 digits.
Further study is needed to determine if it is the underlying condition number
of the flux problem that prevents more digits from being achieved.

{\bf Example 7.}
In the previous example large numbers of GMRES
iterations were needed, in contrast to the Laplace case (cf.\ Example 2).
This motivates studying the growth of iteration count with geometric
complexity. We use exactly the same circular geometries as in Example 3.
Fig.~\ref{f:GMRES}(d) shows that the numbers of iterations are
around ten times those for Laplace, and that they grow with $K$ in a manner
similar to $K^{1/2}$.

%*** MANY VARIANTS:
%There are many other methods than \eqref{pschur} and \eqref{pschur2}
%to work around the fact that the system $QX = C$ is inconsistent,
%and hence create a well-conditioned Schur square system.
%We tested 
%we tried many other variants along the way, including projection of 
%There are many other choices for $V$ other than $CH$ where $H$ are
%densities. We have compared in the Laplace case $V =
%[\mbf{0}_m;\mbf{1}_m;\mbf{0}_m;\mbf{1}_m]$, a smooth function which has $\ww\tr
%V \neq 0$, to $V=CH$ and found little difference in the errors.
%the entries of $V$ must come from
%the discrepancy of a solution to Laplace's equation throughout
%the ball radius $\rp$ (see Sec.~\ref{s:lapempty}), and thus be sufficiently
%smooth. Our piecewise constant choice is the simplest such $V$.

\section{Conclusions}
\label{s:conc}

We have presented a unified framework for the 2nd-kind integral equation
solution of large-scale
periodic Laplace and Stokes BVPs, that we believe will be useful
in related settings such as multiphase composites, photonic crystals, and
particulate flows.
The philosophy is quite simple, and we encourage readers to
try implementing it in their applications:
\bi
\item Use a free-space potential theory representation for
the near neighbor images, plus $\bigO(1)$ auxiliary ``proxy''
degrees of freedom to represent the rest of the infinite lattice.
\item Augment the linear system by applying {\em physical}
boundary conditions on a unit cell.
%avoiding the complication of trying
%to build them into periodic Green's functions or lattice sums.
\item Use a low-rank perturbation of the empty BVP system
to remove the consistency condition,
allowing elimination of the auxiliary unknowns to give a
well-conditioned system %matrix
compatible with fast algorithms.
\ei
% Schur complement to 
%linear algebra
This cures two ills in one fell swoop:
it removes zero eigenvalue(s) that arise in the physical BVP,
and circumvents the issue of non-existence of the
periodic Green's function for %unconstrained
general densities. % arising in potential theory.
%We also prefer integral equation primitive variable (rather than Sherman--Lauricella) formulations, since they will generalize to 3D.

In Section~\ref{s:equiv} we provided a simple unified procedure for periodizing
Laplace and Stokes integral representations
when the physical BVP nullspace is numerically tolerable---%
we believe that this would be a useful starting point for related elliptic BVPs.
However, when a well-conditioned system is desired, we have presented
and analyzed procedures for Laplace
(Section~\ref{s:wellcond}) and Stokes
(Section~\ref{s:stowellcond}) systems,
where the latter exploits the fact that the physical BVP consistency
conditions form a {\em subset} of those of the empty BVP.
%Although for simplicity we did not do so,
The relevant theorems (\ref{t:gary}, \ref{t:alex} and \ref{t:stoalex})
generalize simply to the case of any consistent
right hand side. %(conserving flux or volume).
We showcased the scheme with a high-accuracy computation of
the viscous drag of a square lattice of discs, %(a standard test case),
and in
random composites with a large number of inclusions and several million
unknowns.

\begin{rmk}[Randomization]
In this work we have been explicit about choosing
low-rank matrices $R$ and $H$ proven to have the required
full-rank projections, in both Laplace and Stokes cases
(Lemmata~\ref{l:onesmatrix}, \ref{l:onesdens} and \ref{l:onesdens2}).
%and suspect that they may be optimal.
For other BVPs, this may not be easy or convenient.
However, in practice, there is much flexibility in their choice.
Following \cite{sifuentes2014randomized},
one may even pick random matrices for both $R$ and $H$,
exploiting the fact that with probability one they have 
the required full-rank projections.
In our experiments, the only penalty is a small loss of GMRES convergence rate, and a larger, fluctuating, condition number.
\label{r:random}\end{rmk}

We finish with directions for future work. %that we intend to pursue.
%generalizations
Handling multiphase composites, as in electrostatics \cite{Moura94}
and elastostatics \cite{helsingelasto,cazeaux},
involves only the addition of an integral representation inside inclusions;
the periodization is unchanged.
Moreover, while we restricted our attention to static geometry BVPs, the scheme generalizes to moving geometry problems, such as the flow of bubbles, vesicles or bacteria, in a straightforward manner (e.g., see \cite{periodicp}).   
The scheme generalizes easily to 3D (requiring $M \sim 10^3$ auxiliary unknowns) \cite{garythesis}.
Note that in \eqref{urep}
one need not sum all $3\times 3$ copies of source quadrature nodes
when applying $A$, but just the ones falling inside
the proxy circle (or sphere) \cite{gumerov}.
Exploiting this can lower the overall constant, especially in 3D,
but requires extra bookkeeping in the $C$ matrix block.

By further augmenting the linear system to include decay/radiation conditions,
the scheme has already proven useful for cases when the {\em periodicity is
less than the space dimension}, such as singly-periodic in 2D
\cite{qpfds,periodicp,mlqp} or doubly-periodic in 3D \cite{acper}.
The latter case has applications in electrostatics \cite{lindbolap2per} and
Stokes flow \cite{lindbosto2per,carpets}.
The case of singly-periodic in 3D, and of non-periodic solutions in
periodic geometries, await development.
%mobility problems \cite{mobility}.
%
The generalization to skew and/or high aspect
unit cells has applications in shearing suspensions and in microfluidics.
This is easy in principle, with the proxy circle becoming an oval
in order to ``shield'' the far images, as in Fig.~\ref{f:geom}(c).
However, since $M$ must grow with the aspect ratio, this will require new ideas
beyond an aspect ratio of, say, $10^3$ in 2D.

More theoretical work is required, since our analysis
does not attempt precise bounds on how quadrature errors propagate to
solution errors for the scheme.
Work is also needed to understand whether it is
the underlying conditioning of the BVP,
or an issue with the scheme,
that limits the accuracy to
9 digits at the end of Table~\ref{t:discarray}, and to 4 digits for
Fig.~\ref{f:lapcandy}.
To reduce the large iteration counts for Stokes flows in
complicated close-to-touching geometries, incorporating adaptive schemes
%a mixture of local direct solution
%and in incorporating adaptive schemes for close-to-touching geometries,
such as that of Helsing \cite[Sec.~19]{helsingtut}, or the use of 
fast direct solvers or preconditioners, will be important.
%when the underlying Nystr\"om scheme is capable
%of reaching closer to machine precision.

%Dirichlet BCs for Laplace, w/ unknown potentials.
%Stokes: vesicles, mobility problem CITE Manas.
%Elastostatics.

\section{Acknowledgments}
Were are grateful for useful discussions with Bob Kohn, Leslie Greengard,
Jun Lai, Mark Tygert, and Manas Rachh.
The work of AB and LZ was supported in part by NSF grant
DMS-1216656. The work of LZ was mostly done while at Dartmouth College.
The work of GM and SV was supported in part by NSF grants DMS-1418964 and DMS-1454010.

\bibliographystyle{abbrv} %amsplain}
\bibliography{alex}
\end{document}